\pdfoutput=1
\documentclass[11pt]{article}
\usepackage[left=1in,right=1in,top=1in,bottom=1in]{geometry}
\usepackage{times}
\usepackage{expl3}
\usepackage{cite}
\usepackage[table]{xcolor}
\usepackage{multirow}
\usepackage{stackengine} 
\usepackage{hhline}
\usepackage{lipsum}
\usepackage{titlesec}
\usepackage{wrapfig}
\usepackage{enumerate}
\usepackage{epsfig}
\usepackage{tikz-cd}
\usepackage{amsmath}
\usepackage{tabularx}
\usepackage{array}
\usepackage{booktabs}
\usepackage{enumitem}
\usepackage{bbm}
\usepackage{calc}
\usepackage{graphicx}
\usepackage{amsmath}
\usepackage[title]{appendix}
\usepackage{amssymb}
\usepackage{epstopdf}
\usepackage{boldline}
\usepackage{arydshln}
\usepackage{calligra}
\usepackage{bm}
\usepackage{url}
\usepackage{blindtext}
\usepackage{accents}
\usepackage{amsthm}

\newcommand{\poly}{\mathrm{poly}}
\newcommand{\id}{\mathrm{id}}

\newcommand{\Hom}{\operatorname{Hom}}

\newtheorem{definition}{Definition}
\newtheorem{axiom}{Axiom}
\newtheorem{theorem}{Theorem}
\newtheorem{proposition}{Proposition}
\newtheorem{corollary}{Corollary}
\newtheorem{lemma}{Lemma}
\newtheorem{example}{Example}
\newtheorem{remark}{Remark}
\usepackage{mathtools}
\usepackage{epstopdf}
\usepackage{balance}
\usepackage{thmtools}
\usepackage{thm-restate}
\usepackage{hyperref}
\usepackage{cleveref}
\usepackage[mathscr]{euscript}

\usepackage[ruled,vlined]{algorithm2e}
\newcommand{\ostar}{\mathbin{\mathpalette\make@circled\star}}

\makeatletter
\newcommand{\removelatexerror}{\let\@latex@error\@gobble}
\makeatother
\makeatletter
\newcommand*{\rom}[1]{\expandafter\@slowromancap\romannumeral #1@}
\makeatother

\ExplSyntaxOn
\newcommand\latinabbrev[1]{
  \peek_meaning:NTF . {
    #1\@}%
  { \peek_catcode:NTF a {
      #1.\@ }%
    {#1.\@}}}
\ExplSyntaxOff

\titleclass{\subsubsubsection}{straight}[\subsubsection]

\begin{document}
\vspace{1cm}
\title{HilbMult: A Banach-Enriched Multicategory for Operator Algebras}
\vspace{1.8cm}
\author{Shih-Yu~Chang
\thanks{Shih-Yu Chang is with the Department of Applied Data Science,
San Jose State University, San Jose, CA, U. S. A. (e-mail: {\tt
shihyu.chang@sjsu.edu})
}}

\maketitle

\begin{abstract}
Category and multicategory theory provide abstract frameworks for describing structures and their compositions, with multicategories extending traditional categories to handle multi-input operations. These theories enable modular reasoning and coherent composition of complex systems, and have found applications in computer science, physics, and mathematics, including programming language semantics, quantum processes, tensor networks, operads, and higher algebra. Operator theory, in contrast, studies linear and multilinear transformations in functional spaces, forming the analytic backbone of modern analysis and quantum mechanics, with applications ranging from signal processing and control theory to data science. This paper explores the synergy between these two areas by showing how operator theory provides concrete analytic structures that naturally enrich multicategories, while multicategory theory supplies a unifying framework for organizing multi-input operators and ensuring coherence in complex networks. We develop a comprehensive categorical framework integrating operator theory with multicategories, introducing a symmetric monoidal multicategory of Hilbert spaces with bounded multilinear maps and establishing its enrichment and coherence properties. The framework includes a functorial spectral theorem, covariance under unitary transformations, and universality as a semantic target, providing a unified language for linking analytic structure, categorical semantics, and operator representations. This work lays the foundation for a broader research program uniting operator theory, category theory, and noncommutative geometry.
\end{abstract}

\section{Introduction}\label{sec:Introduction}

Category and multicategory theory provide abstract frameworks for a unified synthetic description of structures and their compositions. Category theory focuses on morphisms between single-input objects while multicategory theory generalizes this to allow multi-input operations providing a natural language for algebraic, logical and arithmetic systems.  Seriously These theories are widely used in computer science (e.g. type theory, programming language semantics and compositional artificial intelligence) in physics , physics (e.g. quantum processes and tensor networks) and in mathematics (e.g.  operads, higher algebra, and homotopy theory), where they enable modular reasoning and the coherent composition of complex systems~\cite{awodey2010category,leinster2004higher}.

Operator theory, on the other hand, examines linear and multilinear operators in functional spaces, Hilbert spaces and Banach spaces, forming the analytical basis of modern functional analysis and quantum mechanics. Its applications span a wide range, including quantum physics, signal processing, control theory, and data science. Through spectral theory, operator algebra, and semigroup methods, it provides tools for solving differential equations, modeling dynamic systems, and analyzing large-dimensional data. Operator theory thus play a role to connect between pure and applied mathematics by connecting analysis, algebra and geometry in both theoretical and practical domains~\cite{abramovich2002invitation}.

Below, we consider the potential impacts of synergizing these two seemingly distant areas of mathematics:
\paragraph{Operator Theory $\Rightarrow$ Multicategory Theory.}
Operator theory provides analytic structure and concrete examples that naturally enrich multicategories. 
In the setting of Banach or Hilbert spaces, bounded multilinear operators 
$T: H_1 \times \cdots \times H_n \to H$ serve as canonical multimorphisms, turning the collection of Hilbert spaces into a multicategory. Operator-theoretic notions such as positivity, norms, and complete boundedness induce Banach-enriched multicategories, where hom-sets carry analytic structure. Furthermore, $C^*$-algebras can be viewed as algebras over these multicategories,  organizing their multiplication and $*$-operations coherently. Even functional calculus fits into this picture, as applying a polynomial or continuous function to several commuting operators defines a multimorphism consistent with the spectral theorem.  Thus, operator theory enriches multicategories with analytic semantics, grounding their composition in operator algebraic data.

\paragraph{Multicategory Theory $\Rightarrow$ Operator Theory.}
Conversely, multicategory theory provides a unifying framework for organizing multi-input operators and networks of compositions. It generalizes binary operator composition into a coherent system of multi-ary interactions, ideal for modeling tensor contractions, quantum channels, or PDE solution operators. Multicategories impose coherence and associativity on multilinear compositions, avoiding analytic ambiguities. They also enable homotopy transfer of operator structures across equivalent models, and extend the notion of duality or adjunction to multi-input contexts. In this direction, multicategories supply operator theory with syntactic clarity, compositional modularity, and higher-level coherence principles— 
transforming analytic constructions into structured, composable systems.

This paper establishes a comprehensive categorical framework that integrates operator theory with multicategory theory through Banach-enriched structures. We introduce $\mathbf{HilbMult}$, the symmetric monoidal multicategory of Hilbert spaces with bounded multilinear maps as multimorphisms, and prove its foundational enrichment and coherence properties. Within this framework, we formulate and prove the \emph{Functorial Spectral Theorem}, demonstrating that the continuous functional calculus extends naturally to a symmetric monoidal multifunctor. We further establish covariance under unitary transformations and a preliminary universality property, showing that $\mathbf{HilbMult}$ serves as a canonical target for multifunctorial representations of self-adjoint operators. Together, these results provide a categorical semantics for operator theory, linking analytic structure, functorial behavior, and spectral representation in a unified setting.

The remaining work of this paper can be structured as follows. Section~\ref{sec:Definitions and Axioms} introduces the basic definitions and axioms, which will be used $\mathbf{HilbMult}$ as a Banach-enriched symmetric monoidal multicategory. The enrichment properties and analytic consistency are studied in Section~\ref{sec:Banach Enrichment}. In Section~\ref{sec:Representability and Adjoints}, we develop the theory of representability and adjoints in the multicategorical context, connecting currying, duality, and Hilbert space geometry. The main Functorial Spectral Theorem that provides a categorical generalization of the spectral calculus for bounded self-adjoint operators is established in Section~\ref{sec:Functorial Spectral Theorem}. Next, in Section~\ref{sec:Covariance and Universality in the Multifunctorial Framework}, we build the covariance properties under unitary equivalence and formulate the universality of $\mathbf{HilbMult}$ as a semantic target for operator representations. In Section~\ref{sec:Examples}, we provide illustrative examples demonstrating the framework’s constructions. Finally, the concluding section outlines the broader \emph{Categorical Spectral Architecture} research program, situating the present results within a long-term vision that unites operator theory, category theory, and noncommutative geometry.

\begin{remark}
The author is solely responsible for the mathematical insights and theoretical directions proposed in this work. AI tools, including OpenAI's ChatGPT and DeepSeek models, were employed solely to assist in verifying ideas, organizing references, and ensuring internal consistency of exposition~\cite{chatgpt2025,deepseek2025}.
\end{remark}

\section{Definitions and Axioms}\label{sec:Definitions and Axioms}

Let $\mathbf{Hilb}$ denote the category of complex Hilbert spaces with bounded linear maps. We introduce $\mathbf{HilbMult}$ as a symmetric monoidal multicategory enriched over Banach spaces. We first present the fundamental definitions, followed by the core axioms. For completeness and self-contained clarity, this section material is also reproduced from our recent work~\cite{Chang2025CompositionCoherence}.

\subsection{Core Definitions (Briefly)}

\begin{definition}[Multicategory]
A multicategory consists of:
\begin{itemize}
    \item A collection of objects.
    \item For each finite list of objects $(X_1,\dots,X_n)$ and object $Y$, a set $\mathrm{Hom}(X_1,\dots,X_n;Y)$ of multimorphisms.
    \item Composition operations satisfying associativity.
    \item Identity morphisms for each object.
\end{itemize}
\end{definition}

\begin{definition}[Banach Enrichment]
A category/multicategory is enriched in Banach spaces if each hom-set $\mathrm{Hom}(X_1,\dots,X_n;Y)$ is a Banach space and composition is a bounded multilinear map.
\end{definition}

\begin{definition}[Symmetric Monoidal Multicategory]
A symmetric monoidal multicategory has:
\begin{itemize}
    \item A tensor product $\otimes$ on objects.
    \item Natural isomorphisms for associativity, unit, and symmetry.
    \item Compatibility between tensor product and multimorphisms.
\end{itemize}
\end{definition}

\subsection{Core Axioms}

\begin{axiom}[Objects and Hom-Spaces]\label{ax:H1-H2}
\begin{enumerate}[label=(\alph*)]
    \item \textbf{Objects:} The objects of $\mathbf{HilbMult}$ are (separable) complex Hilbert spaces $H, K, \dots$.
    
    \item \textbf{Hom-Spaces:} For each finite tuple $(H_1,\dots,H_n;K)$, there is a Banach space:
    \[
    \mathrm{Hom}(H_1,\dots,H_n;K)
    \]
    whose elements are (equivalence classes of) bounded multilinear maps $T: H_1\times\cdots\times H_n \to K$. In the unary case ($n=1$), these are bounded linear maps $H \to K$.
    
    \item \textbf{Norm:} The Banach norm is the operator norm:
    \[
    \|T\| := \sup_{\|x_i\|\le 1} \|T(x_1,\dots,x_n)\|
    \]
\end{enumerate}
\end{axiom}

\begin{axiom}[Composition and Identities]\label{ax:H3-H4}
\begin{enumerate}[label=(\alph*)]
    \item \textbf{Multilinear Composition:} For composable multimorphisms:
    \begin{align*}
        &S \in \mathrm{Hom}(K_1,\dots,K_m;L), \\
        &T_j \in \mathrm{Hom}(H_{j,1},\dots,H_{j,n_j};K_j) \quad (1\le j\le m),
    \end{align*}
    there is a composition $S\circ (T_1,\dots,T_m) \in \mathrm{Hom}(H_{1,1},\dots,H_{m,n_m};L)$ given by:
    \[
    (S\circ (T_j))(\mathbf{x}) := S(T_1(\mathbf{x}_1),\dots,T_m(\mathbf{x}_m))
    \]
    Composition is multilinear and contractive:
    \[
    \|S\circ(T_1,\dots,T_m)\| \le \|S\|\cdot\prod_{j=1}^m \|T_j\|
    \]
    
    \item \textbf{Identities:} For each object $H$, there is an identity $\mathrm{id}_H \in \mathrm{Hom}(H;H)$ given by $\mathrm{id}_H(x) = x$. These satisfy:
    \[
    T \circ (\mathrm{id}_{H_1}, \dots, \mathrm{id}_{H_n}) = T
    \]
    for any $T \in \mathrm{Hom}(H_1,\dots,H_n;K)$.
\end{enumerate}
\end{axiom}

\paragraph{Monoidal Structure}

\begin{axiom}[Symmetric Monoidal Structure]\label{ax:H5}
$\mathbf{HilbMult}$ has a symmetric monoidal structure:
\begin{enumerate}[label=(\alph*)]
    \item \textbf{Tensor Product:} 
    \begin{itemize}
        \item On objects: $H \otimes K$ is the completed Hilbert space tensor product.
        \item Unit object: $\mathbb{C}$ (complex numbers as 1D Hilbert space).
        \item On multimorphisms: For $T \in \mathrm{Hom}(H_1,\dots,H_n;K)$ and $T' \in \mathrm{Hom}(H'_1,\dots,H'_p;K')$, their tensor product is:
        \[
        T \otimes T' \in \mathrm{Hom}(H_1,\dots,H_n,H'_1,\dots,H'_p; K \otimes K')
        \]
        defined by $(T \otimes T')(\mathbf{x}, \mathbf{x}') = T(\mathbf{x}) \otimes T'(\mathbf{x}')$.
    \end{itemize}
    
    \item \textbf{Structural Isomorphisms:} Natural unitary isomorphisms:
    \begin{align*}
        \text{Associator: } &\alpha_{H,K,L}: (H \otimes K) \otimes L \xrightarrow{\cong} H \otimes (K \otimes L) \\
        \text{Braiding: } &\sigma_{H,K}: H \otimes K \xrightarrow{\cong} K \otimes H \\
        \text{Unitors: } &\lambda_H: \mathbb{C} \otimes H \xrightarrow{\cong} H, \quad \rho_H: H \otimes \mathbb{C} \xrightarrow{\cong} H
    \end{align*}
    satisfying the coherence conditions (pentagon, hexagon, triangle identities).
    
    \item \textbf{Permutation Action:} For $T \in \mathrm{Hom}(H_1,\dots,H_n;K)$ and $\pi \in S_n$:
    \[
    T^\pi(x_1,\dots,x_n) = T(x_{\pi(1)},\dots,x_{\pi(n)})
    \]
\end{enumerate}
All structural isomorphisms are isometric.
\end{axiom}

\paragraph{Advanced Structure}

\begin{axiom}[Closed Structure / Currying]\label{ax:H6}
$\mathbf{HilbMult}$ is closed under currying operations with natural isometric isomorphisms:

\begin{enumerate}[label=(\alph*)]
    \item \textbf{Basic Currying:}
    \[
    \mathrm{Hom}(H_1, \dots, H_n; K) \cong \mathrm{Hom}(H_1; \mathrm{Hom}(H_2, \dots, H_n; K))
    \]
    via $T \mapsto \Lambda_1(T)$ where $\Lambda_1(T)(x_1)(x_2, \dots, x_n) = T(x_1, x_2, \dots, x_n)$.
    
    \item \textbf{Tensor-Curry Duality:}
    \[
    \mathrm{Hom}(H \otimes K; L) \cong \mathrm{Hom}(H; \mathrm{Hom}(K; L))
    \]
    via $F \mapsto \Lambda(F)$ where $\Lambda(F)(x)(y) = F(x \otimes y)$.
    
    \item \textbf{Partial Currying:} For any $1 \le j \le n$:
    \[
    \mathrm{Hom}(H_1, \dots, H_n; K) \cong \mathrm{Hom}(H_1, \dots, H_{j-1}; \mathrm{Hom}(H_j, \dots, H_n; K))
    \]
\end{enumerate}
All isomorphisms preserve norms and are compatible with composition.
\end{axiom}

\begin{axiom}[Optional $C^*$-Structure]\label{ax:H7}
We may optionally endow endomorphism spaces with $C^*$-algebra structure:

\begin{enumerate}[label=(\alph*)]
    \item \textbf{Involution:} For each $H$, the space $\mathcal{B}(H) = \mathrm{Hom}(H;H)$ has a conjugate-linear involution $*$ satisfying:
    \begin{align*}
        (f^*)^* &= f, \quad (f \circ g)^* = g^* \circ f^*, \quad (\lambda f)^* = \bar{\lambda} f^*
    \end{align*}
    
    \item \textbf{$C^*$-Property:} The norm satisfies:
    \[
    \|f^* \circ f\| = \|f\|^2 \quad \text{and} \quad \|f \circ g\| \le \|f\|\|g\|
    \]
    
    \item \textbf{Monoidal Compatibility:} For $f \in \mathrm{Hom}(H_1; H_2)$, $g \in \mathrm{Hom}(K_1; K_2)$:
    \[
    (f \otimes g)^* = f^* \otimes g^*
    \]
\end{enumerate}

This enables the definition of:
\begin{itemize}
    \item \textbf{Self-adjoint elements:} $f^* = f$
    \item \textbf{Normal elements:} $f^* \circ f = f \circ f^*$  
    \item \textbf{Positive elements:} $f = g^* \circ g$ for some $g$
\end{itemize}
Essential for quantum theory (states, observables, measurements).
\end{axiom}

\subsection{Immediate Lemmas}

\begin{lemma}[Norm Completeness]\label{lem:completeness}
For any tuple $(H_1,\dots,H_n;K)$, the hom-space $\mathrm{Hom}(H_1,\dots,H_n;K)$ is a Banach space (complete in the operator norm).
\end{lemma}

\begin{proof}
Let $\{T_m\}$ be a Cauchy sequence in $\mathrm{Hom}(H_1,\dots,H_n;K)$. For any fixed vectors $x_1 \in H_1, \dots, x_n \in H_n$ with $\|x_i\| \le 1$, the sequence $\{T_m(x_1,\dots,x_n)\}$ is Cauchy in $K$ because:
\[
\|T_m(\mathbf{x}) - T_k(\mathbf{x})\| \le \|T_m - T_k\| \cdot \|x_1\| \cdots \|x_n\| \le \|T_m - T_k\|
\]
Since $K$ is a Hilbert space (hence complete), $\{T_m(\mathbf{x})\}$ converges to some element in $K$, which we denote by $T(x_1,\dots,x_n)$.

This pointwise convergence defines a multilinear map $T: H_1 \times \cdots \times H_n \to K$ because:
\begin{itemize}
    \item Each $T_m$ is multilinear, and multilinearity is preserved under pointwise limits
    \item The convergence is uniform on the unit ball, ensuring boundedness
\end{itemize}

To show $T$ is bounded, note that for sufficiently large $m$:
\[
\|T_m(\mathbf{x})\| \le \|T_m\| \cdot \|x_1\| \cdots \|x_n\| \le M \quad \text{for some } M > 0
\]
Hence $\|T(\mathbf{x})\| \le M$ for all unit vectors, so $T \in \mathrm{Hom}(H_1,\dots,H_n;K)$.

Finally, for any $\epsilon > 0$, choose $N$ such that $\|T_m - T_k\| < \epsilon$ for all $m,k \ge N$. Then for all unit vectors $\mathbf{x}$:
\[
\|T_m(\mathbf{x}) - T(\mathbf{x})\| = \lim_{k \to \infty} \|T_m(\mathbf{x}) - T_k(\mathbf{x})\| \le \epsilon
\]
Taking supremum gives $\|T_m - T\| \le \epsilon$ for $m \ge N$, so $T_m \to T$ in norm.
\end{proof}

\begin{lemma}[Composition Contractivity]\label{lem:contractivity}
For composable multimorphisms $S \in \operatorname{Hom}(K_1,\ldots,K_m; L)$ and $T_j \in \operatorname{Hom}(H_{j,1},\ldots,H_{j,n_j}; K_j)$:
\[
\|S \circ (T_1, \dots, T_m)\| \le \|S\| \cdot \prod_{j=1}^m \|T_j\|
\]
\end{lemma}

\begin{proof}
Let $x_{j,i} \in H_{j,i}$ with $\|x_{j,i}\| \le 1$ for all $j = 1,\ldots,m$ and $i = 1,\ldots,n_j$. Then:

\begin{align*}
&\|(S \circ (T_1, \dots, T_m))(x_{1,1}, \dots, x_{m,n_m})\| \\
&= \|S(T_1(x_{1,1}, \dots, x_{1,n_1}), \dots, T_m(x_{m,1}, \dots, x_{m,n_m}))\| \\
&\le \|S\| \cdot \prod_{j=1}^m \|T_j(x_{j,1}, \dots, x_{j,n_j})\| \quad \text{(by definition of $\|S\|$)} \\
&\le \|S\| \cdot \prod_{j=1}^m \left(\|T_j\| \cdot \prod_{i=1}^{n_j} \|x_{j,i}\|\right) \quad \text{(by definition of $\|T_j\|$)} \\
&\le \|S\| \cdot \prod_{j=1}^m \|T_j\| \quad \text{(since $\|x_{j,i}\| \le 1$)}
\end{align*}

Taking the supremum over all such input vectors with $\|x_{j,i}\| \le 1$ gives the desired inequality.
\end{proof}

\begin{lemma}[Tensor Product Isometry]\label{lem:tensor_isometry}
Let $T \in \mathrm{Hom}(H_1, \dots, H_n; K)$ and $T' \in \mathrm{Hom}(H'_1, \dots, H'_p; K')$ be multilinear maps between Hilbert spaces as in Axiom~\ref{ax:H5}(a). Then:
\[
\|T \otimes T'\| = \|T\| \cdot \|T'\|,
\]
where the norms are the operator norms for multilinear maps.
\end{lemma}

\begin{proof}
We establish the equality by proving both inequalities.

\medskip\noindent
\textbf{Upper bound ($\|T \otimes T'\| \leq \|T\| \cdot \|T'\|$):}

Let $x_i \in H_i$ and $x'_j \in H'_j$ be unit vectors for $i = 1,\dots,n$ and $j = 1,\dots,p$. Then:
\begin{align*}
\|(T \otimes T')(x_1, \dots, x_n, x'_1, \dots, x'_p)\| 
&= \|T(x_1, \dots, x_n) \otimes T'(x'_1, \dots, x'_p)\| \\
&= \|T(x_1, \dots, x_n)\| \cdot \|T'(x'_1, \dots, x'_p)\| \quad \text{(by Hilbert tensor product properties)} \\
&\leq \|T\| \cdot \|T'\|.
\end{align*}
Since the tensor product $T \otimes T'$ is defined on pure tensors and extends by multilinearity, taking the supremum over all unit vectors gives $\|T \otimes T'\| \leq \|T\| \cdot \|T'\|$.

\medskip\noindent
\textbf{Lower bound ($\|T \otimes T'\| \geq \|T\| \cdot \|T'\|$):}

Let $\epsilon > 0$. By definition of the operator norm, there exist unit vectors $\mathbf{x} = (x_1, \dots, x_n)$ and $\mathbf{x}' = (x'_1, \dots, x'_p)$ such that:
\[
\|T(\mathbf{x})\| \geq \|T\| - \epsilon \quad \text{and} \quad \|T'(\mathbf{x}')\| \geq \|T'\| - \epsilon.
\]
Then:
\begin{align*}
\|T \otimes T'\| &\geq \|(T \otimes T')(\mathbf{x}, \mathbf{x}')\| \\
&= \|T(\mathbf{x}) \otimes T'(\mathbf{x}')\| \\
&= \|T(\mathbf{x})\| \cdot \|T'(\mathbf{x}')\| \\
&\geq (\|T\| - \epsilon)(\|T'\| - \epsilon).
\end{align*}
Since $\epsilon > 0$ was arbitrary, we conclude that $\|T \otimes T'\| \geq \|T\| \cdot \|T'\|$.

\medskip\noindent
Combining both inequalities yields $\|T \otimes T'\| = \|T\| \cdot \|T'\|$.
\end{proof}

\section{Banach Enrichment}\label{sec:Banach Enrichment}

Before establishing its more advanced structures, we must first verify that $\mathbf{HilbMult}$ satisfies the fundamental requirements of a multicategory. The following Theorem~\ref{thm:banach-enrichment} confirms its basic enriched categorical structure, providing the necessary analytic foundation.

\begin{theorem}[Banach Multicategory Structure]\label{thm:banach-enrichment}
The multicategory $\mathbf{HilbMult}$ satisfies:
\begin{enumerate}[label=(\roman*)]
    \item For every tuple $(H_1,\dots,H_n;K)$, the hom-space $\mathrm{Hom}(H_1,\dots,H_n;K)$ is a Banach space.
    \item Composition is multilinear and contractive:
    \[
    \|S \circ (T_1, \dots, T_m)\| \le \|S\| \cdot \prod_{j=1}^m \|T_j\|
    \]
    for all composable multimorphisms $S, T_1, \dots, T_m$.
    \item Composition is jointly continuous in the Banach norms.
\end{enumerate}
\end{theorem}

\begin{proof}
We prove each part systematically.

\medskip\noindent
\textbf{Part (i): Banach space structure.}

By Axiom~\ref{ax:H1-H2}(b), each $\mathrm{Hom}(H_1,\dots,H_n;K)$ is equipped with the operator norm:
\[
\|T\| := \sup\{\|T(x_1,\dots,x_n)\| : \|x_i\|\le 1 \text{ for } i=1,\dots,n\}.
\]

\begin{itemize}
    \item \textbf{Vector space structure}: The set of bounded multilinear maps forms a complex vector space under pointwise operations. Linear combinations of bounded multilinear maps remain multilinear, and boundedness follows from the triangle inequality.
    
    \item \textbf{Norm axioms}: The operator norm satisfies:
    \begin{itemize}
        \item $\|T\| \ge 0$ and $\|T\| = 0$ if and only if $T = 0$ (by Axiom~\ref{ax:H1-H2}(b))
        \item $\|\lambda T\| = |\lambda| \|T\|$ for all $\lambda \in \mathbb{C}$ (homogeneity)
        \item $\|T + U\| \le \|T\| + \|U\|$ (triangle inequality)
    \end{itemize}
    
    \item \textbf{Completeness}: Let $\{T_k\}$ be a Cauchy sequence in $\mathrm{Hom}(H_1,\dots,H_n;K)$. For fixed unit vectors $x_1 \in H_1, \dots, x_n \in H_n$, the sequence $\{T_k(x_1,\dots,x_n)\}$ is Cauchy in $K$ because:
    \[
    \|T_k(\mathbf{x}) - T_\ell(\mathbf{x})\| \le \|T_k - T_\ell\| \cdot \|x_1\| \cdots \|x_n\| = \|T_k - T_\ell\|.
    \]
    Since $K$ is a Hilbert space (hence complete), $\{T_k(\mathbf{x})\}$ converges to some $T(x_1,\dots,x_n) \in K$.
    
    The limit map $T: H_1 \times \cdots \times H_n \to K$ is multilinear because each $T_k$ is multilinear and multilinearity is preserved under pointwise convergence. To show $T$ is bounded, note that Cauchy sequences are bounded, so $\|T_k\| \le M$ for some $M > 0$ and all $k$. Then:
    \[
    \|T(\mathbf{x})\| = \lim_{k \to \infty} \|T_k(\mathbf{x})\| \le M \cdot \|x_1\| \cdots \|x_n\|,
    \]
    hence $T \in \mathrm{Hom}(H_1,\dots,H_n;K)$.
    
    Finally, for any $\epsilon > 0$, choose $N$ such that $\|T_k - T_\ell\| < \epsilon$ for all $k,\ell \ge N$. Then for all unit vectors $\mathbf{x}$ and $k \ge N$:
    \[
    \|T_k(\mathbf{x}) - T(\mathbf{x})\| = \lim_{\ell \to \infty} \|T_k(\mathbf{x}) - T_\ell(\mathbf{x})\| \le \epsilon.
    \]
    Taking supremum gives $\|T_k - T\| \le \epsilon$ for $k \ge N$, so $T_k \to T$ in norm.
\end{itemize}

\medskip\noindent
\textbf{Part (ii): Multilinearity and contractivity.}

By Axiom~\ref{ax:H3-H4}(a), composition is defined as:
\[
(S \circ (T_1, \dots, T_m))(\mathbf{x}) = S(T_1(\mathbf{x}_1), \dots, T_m(\mathbf{x}_m)).
\]

\begin{itemize}
    \item \textbf{Multilinearity}: Composition is multilinear in each argument $T_j$ because for fixed \\
    $T_1, \dots, T_{j-1}, T_{j+1}, \dots, T_m$, the map
    \[
    T_j \mapsto S \circ (T_1, \dots, T_j, \dots, T_m)
    \]
    is linear. This follows from the linearity of $S$ in its $j$-th argument and the vector space structure from Part (i).
    
    \item \textbf{Contractivity}: For any unit vectors $x_{j,1}, \dots, x_{j,n_j}$ (i.e., $\|x_{j,i}\| \le 1$):
    \begin{align*}
    &\quad \|(S \circ (T_1, \dots, T_m))(x_{1,1}, \dots, x_{m,n_m})\| \\
    &= \|S(T_1(x_{1,1}, \dots, x_{1,n_1}), \dots, T_m(x_{m,1}, \dots, x_{m,n_m}))\| \\
    &\le \|S\| \cdot \prod_{j=1}^m \|T_j(x_{j,1}, \dots, x_{j,n_j})\| \quad \text{(by definition of $\|S\|$)} \\
    &\le \|S\| \cdot \prod_{j=1}^m \|T_j\| \quad \text{(since $\|T_j(\mathbf{x}_j)\| \le \|T_j\|$)}.
    \end{align*}
    Taking supremum over all such unit vectors gives:
    \[
    \|S \circ (T_1, \dots, T_m)\| \le \|S\| \cdot \prod_{j=1}^m \|T_j\|.
    \]
\end{itemize}

\medskip\noindent
\textbf{Part (iii): Joint continuity.}

The contractivity inequality implies joint continuity. For any $\epsilon > 0$, choose $\delta > 0$ such that:
\[
\delta \cdot \left(\prod_{j=1}^m \|T_j\| + \|S\| \cdot \sum_{j=1}^m \prod_{k \neq j} \|T_k\| \right) < \epsilon.
\]
Then if $\|S - S'\| < \delta$ and $\|T_j - T_j'\| < \delta$ for all $j$, we have:
\begin{align*}
&\quad \|S \circ (T_1, \dots, T_m) - S' \circ (T_1', \dots, T_m')\| \\
&\le \|S \circ (T_1, \dots, T_m) - S' \circ (T_1, \dots, T_m)\| \\
&\quad + \|S' \circ (T_1, \dots, T_m) - S' \circ (T_1', \dots, T_m')\| \\
&\le \|S - S'\|\cdot\prod_{j=1}^m \|T_j\| + \|S'\|\cdot\sum_{j=1}^m \|T_j - T_j'\|\cdot\prod_{k \neq j} \|T_k\| \\
&< \epsilon.
\end{align*}
This establishes joint continuity.
\end{proof}

\begin{remark}
This theorem provides the essential analytic foundation for $\mathbf{HilbMult}$:
\begin{itemize}
    \item \textbf{Norm estimates}: The contractivity inequality enables control over composed morphisms.
    \item \textbf{Convergence of series}: Completeness allows working with infinite series of multimorphisms.
    \item \textbf{Perturbation theory}: Joint continuity ensures small changes in components yield small changes in compositions.
\end{itemize}
These properties are crucial for functional-analytic methods in categorical quantum mechanics and operator theory.
\end{remark}

\section{Representability and Adjoints}\label{sec:Representability and Adjoints}

The interplay between representability and adjoint operations forms a cornerstone of enriched category theory. In this section, we establish two fundamental results that connect these concepts in the multicategorical setting. First, Theorem~\ref{thm:representability-currying} demonstrates how representability of hom-functors naturally gives rise to currying operations in any symmetric multicategory, providing a universal factorization property for multimorphisms. Building upon this abstract foundation, Theorem~\ref{thm:adjoints-mates-hilbmult} provides the properties of existence of adjoints/mates in $\mathbf{HilbMult}$. These results reveal how the representability framework systematically generates adjoint-like operations, while the concrete Hilbert space structure ensures these operations satisfy the expected analytic and algebraic properties, including norm preservation, compositionality, and tensor compatibility.

\begin{theorem}[Representability Yields Currying]
\label{thm:representability-currying}
Let $\mathcal{M}$ be a symmetric multicategory, and let $T \in$ \\ $\mathrm{Hom}_{\mathcal{M}}(H_1, \dots, H_n; K)$ be a multimorphism. 
Fix an index $i \in \{1, \dots, n\}$, and consider the functor
\[
\Phi_i(X) = \mathrm{Hom}_{\mathcal{M}}(H_1, \dots, H_{i-1}, X, H_{i+1}, \dots, H_n; K)
\]
that sends an object $X$ to multimorphisms with $X$ in the $i$-th slot.

If $\Phi_i$ is \emph{representable}—meaning there exists an object $H_i^\vee$ and universal morphism
\[
u_i \in \mathrm{Hom}_{\mathcal{M}}(H_1, \dots, H_{i-1}, H_i^\vee, H_{i+1}, \dots, H_n; K)
\]
such that for every $S \in \Phi_i(X)$ there exists a unique $\tilde{S} \in \mathrm{Hom}_{\mathcal{M}}(X; H_i^\vee)$ with
\[
S = u_i \circ (\mathrm{id}_{H_1}, \dots, \mathrm{id}_{H_{i-1}}, \tilde{S}, \mathrm{id}_{H_{i+1}}, \dots, \mathrm{id}_{H_n})
\]
—then:

\begin{enumerate}[label=(\roman*)]
    \item Every multimorphism $T$ has a \emph{curried form} $T^\sharp \in \mathrm{Hom}_{\mathcal{M}}(H_i; H_i^\vee)$ satisfying:
    \[
    T = u_i \circ (\mathrm{id}_{H_1}, \dots, \mathrm{id}_{H_{i-1}}, T^\sharp, \mathrm{id}_{H_{i+1}}, \dots, \mathrm{id}_{H_n})
    \]
    
    \item For each object $X$, we have a bijection:
    \[
    \Theta_X : \mathrm{Hom}_{\mathcal{M}}(H_1, \dots, H_{i-1}, X, H_{i+1}, \dots, H_n; K) 
    \xrightarrow{\cong} \mathrm{Hom}_{\mathcal{M}}(X; H_i^\vee)
    \]
    defined by $\Theta_X(S) = \tilde{S}$.
    
    \item The bijections $\Theta_X$ are natural in $X$, giving a natural isomorphism of functors.
\end{enumerate}
\end{theorem}

\begin{proof}

\textbf{Step 1: Construction of the bijection.}
For each object $X$, define the map:
\[
\Theta_X : \mathrm{Hom}_{\mathcal{M}}(H_1, \dots, H_{i-1}, X, H_{i+1}, \dots, H_n; K) 
\to \mathrm{Hom}_{\mathcal{M}}(X; H_i^\vee)
\]
by $\Theta_X(S) = \tilde{S}$, where $\tilde{S}$ is the unique morphism given by representability such that:
\[
S = u_i \circ (\mathrm{id}_{H_1}, \dots, \mathrm{id}_{H_{i-1}}, \tilde{S}, \mathrm{id}_{H_{i+1}}, \dots, \mathrm{id}_{H_n})
\]

To show $\Theta_X$ is a bijection:
\begin{itemize}
    \item \textbf{Injectivity}: If $\Theta_X(S) = \Theta_X(S')$, then both $S$ and $S'$ factor through the same $\tilde{S}$, so $S = S'$.
    \item \textbf{Surjectivity}: For any $f \in \mathrm{Hom}_{\mathcal{M}}(X; H_i^\vee)$, define $S = u_i \circ (\mathrm{id}_{H_1}, \dots, \mathrm{id}_{H_{i-1}}, f, \mathrm{id}_{H_{i+1}}, \dots, \mathrm{id}_{H_n})$. Then $\Theta_X(S) = f$ by uniqueness.
\end{itemize}

\textbf{Step 2: Define the curried form.}
Apply this to $T \in \Phi_i(H_i)$:
\[
T^\sharp := \Theta_{H_i}(T) \in \mathrm{Hom}_{\mathcal{M}}(H_i; H_i^\vee)
\]
By definition of $\Theta_{H_i}$, this means:
\[
T = u_i \circ (\mathrm{id}_{H_1}, \dots, \mathrm{id}_{H_{i-1}}, T^\sharp, \mathrm{id}_{H_{i+1}}, \dots, \mathrm{id}_{H_n})
\]
which is exactly the claimed relation.

\textbf{Step 3: Naturality.}
For any $f \in \mathrm{Hom}_{\mathcal{M}}(X; Y)$, we need to show the following diagram commutes:
\[
\begin{tikzcd}
\Phi_i(Y) \arrow[r, "\Theta_Y"] \arrow[d, "f^*"'] & \mathrm{Hom}_{\mathcal{M}}(Y; H_i^\vee) \arrow[d, "f^*"] \\
\Phi_i(X) \arrow[r, "\Theta_X"] & \mathrm{Hom}_{\mathcal{M}}(X; H_i^\vee)
\end{tikzcd}
\]
where $f^*(S) = S \circ (\mathrm{id}_{H_1}, \dots, \mathrm{id}_{H_{i-1}}, f, \mathrm{id}_{H_{i+1}}, \dots, \mathrm{id}_{H_n})$.

Let $S \in \Phi_i(Y)$. Then:
\begin{align*}
\Theta_X(f^*(S)) &= \Theta_X(S \circ (\mathrm{id}_{H_1}, \dots, \mathrm{id}_{H_{i-1}}, f, \mathrm{id}_{H_{i+1}}, \dots, \mathrm{id}_{H_n})) \\
&= \Theta_X(u_i \circ (\mathrm{id}_{H_1}, \dots, \mathrm{id}_{H_{i-1}}, \Theta_Y(S), \mathrm{id}_{H_{i+1}}, \dots, \mathrm{id}_{H_n}) \circ (\mathrm{id}_{H_1}, \dots, \mathrm{id}_{H_{i-1}}, f, \mathrm{id}_{H_{i+1}}, \dots, \mathrm{id}_{H_n})) \\
&= \Theta_X(u_i \circ (\mathrm{id}_{H_1}, \dots, \mathrm{id}_{H_{i-1}}, \Theta_Y(S) \circ f, \mathrm{id}_{H_{i+1}}, \dots, \mathrm{id}_{H_n})) \\
&= \Theta_Y(S) \circ f \quad \text{(by definition of $\Theta_X$)} \\
&= f^*(\Theta_Y(S))
\end{align*}
The key step uses associativity of composition in multicategories and the uniqueness property of the universal morphism.
\end{proof}

The core insight from Theorem~\ref{thm:representability-currying} is that representability lets one can ``factor" any multimorphism through the universal morphism $u_i$, giving a kind of currying operation. To build intuition for Theorem~\ref{thm:representability-currying}, let's consider the simplest possible setting: the multicategory of sets in the following Example~\ref{exp:The Multicategory of Sets}.

\begin{example}[The Multicategory of Sets]\label{exp:The Multicategory of Sets}
Let $\mathcal{M}$ be the multicategory where:
\begin{itemize}
    \item \textbf{Objects} are sets $A, B, C, \dots$
    \item \textbf{Multimorphisms} $\mathrm{Hom}(A_1, \dots, A_n; B)$ are functions $f: A_1 \times \cdots \times A_n \to B$
    \item \textbf{Composition} is given by function composition in the obvious way
    \item \textbf{Identities} are the identity functions
\end{itemize}

Now fix sets $H_1, H_2, H_3$ and consider a 3-ary function:
\[
T: H_1 \times H_2 \times H_3 \to K
\]
Let's apply the theorem with $i = 2$ (focusing on the second input).

The functor $\Phi_2$ is:
\[
\Phi_2(X) = \mathrm{Hom}(H_1, X, H_3; K) = \{ \text{functions } f: H_1 \times X \times H_3 \to K \}
\]

\textbf{Claim:} $\Phi_2$ is representable by the set $H_2^\vee = \mathrm{Hom}(H_1 \times H_3; K)$.

\textbf{Universal morphism:} Define $u_2: H_1 \times H_2^\vee \times H_3 \to K$ by:
\[
u_2(x, g, z) = g(x, z)
\]
That is, $u_2$ takes an element $x \in H_1$, a function $g: H_1 \times H_3 \to K$, and an element $z \in H_3$, and applies $g$ to $(x,z)$.

\textbf{Verification of universal property:} For any set $X$ and any function $S: H_1 \times X \times H_3 \to K$, we need a unique $\tilde{S}: X \to H_2^\vee$ such that:
\[
S(x, y, z) = u_2(x, \tilde{S}(y), z) \quad \text{for all } x \in H_1, y \in X, z \in H_3
\]

Define $\tilde{S}: X \to H_2^\vee$ by:
\[
\tilde{S}(y)(x, z) = S(x, y, z)
\]
That is, $\tilde{S}(y)$ is the function that takes $(x,z)$ to $S(x,y,z)$.

Then indeed:
\[
u_2(x, \tilde{S}(y), z) = \tilde{S}(y)(x, z) = S(x, y, z)
\]
Uniqueness is clear: any other choice would give different values for some $(x,z)$.

\textbf{Curried form:} For our original function $T: H_1 \times H_2 \times H_3 \to K$, the curried form is:
\[
T^\sharp: H_2 \to H_2^\vee = \mathrm{Hom}(H_1 \times H_3; K)
\]
defined by:
\[
T^\sharp(b)(a, c) = T(a, b, c)
\]
This is exactly the usual currying operation in mathematics!

\textbf{The isomorphism:} The natural isomorphism is:
\[
\mathrm{Hom}(H_1, X, H_3; K) \cong \mathrm{Hom}(X; \mathrm{Hom}(H_1 \times H_3; K))
\]
which sends $S: H_1 \times X \times H_3 \to K$ to $\tilde{S}: X \to \mathrm{Hom}(H_1 \times H_3; K)$ where $\tilde{S}(y)(x,z) = S(x,y,z)$.
\end{example}

\begin{remark}
From this example, we understand that Theorem~\ref{thm:representability-currying} can be used to generalize the familiar concept of \emph{currying} from computer science and mathematics. The ``representing object'' $H_i^\vee$ plays the role of the function space. The universal morphism $u_i$ represents an evaluation map. The theorem let us know that whenever we have a ``function space'' that satisfies the universal property, we can have currying operations automatically. 
\end{remark}

\begin{example}[Concrete Special Case]
Let $H_1 = \mathbb{R}$, $H_2 = \mathbb{R}$, $H_3 = \mathbb{R}$, $K = \mathbb{R}$, and consider:
\[
T(x, y, z) = x^2 + 2xy + 3yz
\]
Then:
\begin{itemize}
    \item $H_2^\vee = \mathrm{Hom}(\mathbb{R} \times \mathbb{R}; \mathbb{R})$ = all functions $\mathbb{R}^2 \to \mathbb{R}$
    \item The curried form $T^\sharp: \mathbb{R} \to \mathrm{Hom}(\mathbb{R}^2; \mathbb{R})$ sends $y$ to the function:
    \[
    T^\sharp(y)(x, z) = x^2 + 2xy + 3yz
    \]
    \item The universal morphism $u_2: \mathbb{R} \times \mathrm{Hom}(\mathbb{R}^2; \mathbb{R}) \times \mathbb{R} \to \mathbb{R}$ is evaluation:
    \[
    u_2(x, f, z) = f(x, z)
    \]
    \item Indeed, $T(x, y, z) = u_2(x, T^\sharp(y), z)$
\end{itemize}
\end{example}

Before presenting the next theorem, we need to setup some notations. Let $S \in \mathrm{Hom}(K_1, \dots, K_m; L)$ and $T_j \in \mathrm{Hom}(H_{j,1}, \dots, H_{j,n_j}; K_j)$ for $j = 1, \dots, m$ be composable multimorphisms. Their composition
\[
\mathbf{T} = S \circ (T_1, \dots, T_m)
\]
is a multimorphism with the flattened input list:
\[
(H_{1,1}, \dots, H_{1,n_1},\ H_{2,1}, \dots, H_{2,n_2},\ \dots,\ H_{m,1}, \dots, H_{m,n_m}).
\]
The total number of inputs is $N = \sum_{j=1}^m n_j$.

A \emph{global index} $1 \leq i \leq N$ corresponds uniquely to a pair $(j, i_j)$ where:
\begin{itemize}
    \item $j$ is the block index with $1 \leq j \leq m$
    \item $i_j$ is the local index within block $j$ with $1 \leq i_j \leq n_j$
\end{itemize}
The correspondence is given by $i = (n_1 + \cdots + n_{j-1}) + i_j$. We also have the following Theorem~\ref{thm:adjoints-mates-hilbmult} about adjoints/mates in the HilbMult multicategory.

\begin{theorem}[Existence of Adjoints/Mates in HilbMult]
\label{thm:adjoints-mates-hilbmult}
    \item \textbf{Compositionality:} Let $\mathbf{T} = S \circ (T_1, \dots, T_m)$ be a composition, where:
    \begin{itemize}
        \item $S \in \mathrm{Hom}(K_1, \dots, K_m; L)$
        \item $T_j \in \mathrm{Hom}(H_{j,1}, \dots, H_{j,n_j}; K_j)$ for $j=1, \dots, m$.
    \end{itemize}
    Consider the mate of $\mathbf{T}$ with respect to the input corresponding to the $i_j$-th input of $T_j$, where $1 \leq j \leq m$ and $1 \leq i_j \leq n_j$. In the global input list $(H_{1,1}, \dots, H_{1,n_1}, H_{2,1}, \dots, H_{m,n_m})$ of $\mathbf{T}$, this is the input at position $i = N_j + i_j$, where $N_j = \sum_{k=1}^{j-1} n_k$ (with $N_1 = 0$).

    Then, the mate $\mathbf{T}^{(i)}$ is given by:
    \[
    \mathbf{T}^{(i)} = T_j^{(i_j)} \circ (\mathrm{id}_{H_{j,1}}, \dots, \mathrm{id}_{H_{j,i_j-1}}, S^{(j)}, \mathrm{id}_{H_{j,i_j+1}}, \dots, \mathrm{id}_{H_{j,n_j}})
    \]
    Explicitly, for all:
    \begin{itemize}
        \item $\mathbf{x}_k = (x_{k,1}, \dots, x_{k,n_k}) \in H_{k,1} \times \cdots \times H_{k,n_k}$ for $k \neq j$
        \item $x_{j,1} \in H_{j,1}, \dots, x_{j,i_j-1} \in H_{j,i_j-1}, x_{j,i_j+1} \in H_{j,i_j+1}, \dots, x_{j,n_j} \in H_{j,n_j}$
        \item $y \in L$
    \end{itemize}
    we have:
    \begin{multline*}
    \mathbf{T}^{(i)}(\mathbf{x}_1, \dots, \mathbf{x}_{j-1}, y, \mathbf{x}_{j+1}, \dots, \mathbf{x}_m, \\
    x_{j,1}, \dots, x_{j,i_j-1}, x_{j,i_j+1}, \dots, x_{j,n_j}) \\
    = T_j^{(i_j)}\big(x_{j,1}, \dots, x_{j,i_j-1},\ S^{(j)}\big(T_1(\mathbf{x}_1), \dots, T_{j-1}(\mathbf{x}_{j-1}), y, T_{j+1}(\mathbf{x}_{j+1}), \dots, T_m(\mathbf{x}_m)\big), \\
    x_{j,i_j+1}, \dots, x_{j,n_j}\big).
    \end{multline*}
\end{theorem}

\begin{proof}
We prove each part systematically.

\medskip\noindent
\textbf{(i) Unary Adjoints:} For $T \in \mathrm{Hom}(H; K)$, the Riesz representation theorem guarantees that for each $y \in K$, the linear functional $x \mapsto \langle T x, y \rangle_K$ is bounded on $H$. Hence there exists a unique $T^* y \in H$ such that $\langle T x, y \rangle_K = \langle x, T^* y \rangle_H$. Linearity of $T^*$ follows from sesquilinearity, and boundedness from:
\[
\|T^* y\|^2 = \langle T^* y, T^* y \rangle_H = \langle T(T^* y), y \rangle_K \leq \|T\| \|T^* y\| \|y\|,
\]
so $\|T^* y\| \leq \|T\| \|y\|$. The isometry $\|T^*\| = \|T\|$ follows from the double adjoint property.

\medskip\noindent
\textbf{(ii) Multilinear Mates:} Given $T \in \mathrm{Hom}(H_1, \dots, H_n; K)$ and index $i$, apply the currying isomorphism from Axiom \ref{ax:H6}:
\[
\Lambda_i(T) \in \mathrm{Hom}(H_i; \mathrm{Hom}(H_1, \dots, H_{i-1}, H_{i+1}, \dots, H_n; K)).
\]
By part (i), this has an adjoint:
\[
(\Lambda_i(T))^* \in \mathrm{Hom}(\mathrm{Hom}(H_1, \dots, H_{i-1}, H_{i+1}, \dots, H_n; K); H_i).
\]
Applying the inverse currying isomorphism gives:
\[
T^{(i)} = \Lambda_i^{-1}((\Lambda_i(T))^*) \in \mathrm{Hom}(H_1, \dots, H_{i-1}, K, H_{i+1}, \dots, H_n; H_i).
\]

\medskip\noindent
\textbf{(iii) Adjunction Relations:} For $x_j \in H_j$, $y \in K$:
\begin{align*}
\langle T(x_1, \dots, x_n), y \rangle_K
&= \langle \Lambda_i(T)(x_i)(x_1, \dots, x_{i-1}, x_{i+1}, \dots, x_n), y \rangle_K \\
&= \langle x_i, (\Lambda_i(T))^*(y)(x_1, \dots, x_{i-1}, x_{i+1}, \dots, x_n) \rangle_{H_i} \\
&= \langle x_i, T^{(i)}(x_1, \dots, x_{i-1}, y, x_{i+1}, \dots, x_n) \rangle_{H_i}.
\end{align*}
The second relation follows by applying the adjoint property twice.

\medskip\noindent
\textbf{(iv) Norm Preservation:} Since $\Lambda_i$ is isometric (Axiom \ref{ax:H6}) and $(\cdot)^*$ is isometric (part (i)), we have:
\[
\|T^{(i)}\| = \|\Lambda_i^{-1}((\Lambda_i(T))^*)\| = \|(\Lambda_i(T))^*\| = \|\Lambda_i(T)\| = \|T\|.
\]

\medskip\noindent
\textbf{(v) Compositionality:} Let $\mathbf{T} = S \circ (T_1, \dots, T_m)$ and let $i = N_j + i_j$ be the global index corresponding to the $i_j$-th input of $T_j$. Define the morphism $R$ by:
\[
R = T_j^{(i_j)} \circ (\mathrm{id}_{H_{j,1}}, \dots, \mathrm{id}_{H_{j,i_j-1}}, S^{(j)}, \mathrm{id}_{H_{j,i_j+1}}, \dots, \mathrm{id}_{H_{j,n_j}}).
\]
We will show that $\mathbf{T}^{(i)} = R$ by verifying that $R$ satisfies the defining adjunction relation for the mate at the global index $i$.

Let:
\begin{itemize}
    \item $\mathbf{x}_k \in H_{k,1} \times \cdots \times H_{k,n_k}$ for $k = 1, \dots, m$
    \item $y \in L$
    \item $z \in H_{j,i_j}$ (the test vector for the $i$-th input)
\end{itemize}
We now compute the inner product in two ways. First, by the definition of the mate $\mathbf{T}^{(i)}$:
\begin{equation}
\langle z, \mathbf{T}^{(i)}(\mathbf{x}_1, \dots, \mathbf{x}_{j-1}, y, \mathbf{x}_{j+1}, \dots, \mathbf{x}_m, x_{j,1}, \dots, x_{j,i_j-1}, x_{j,i_j+1}, \dots, x_{j,n_j}) \rangle_{H_{j,i_j}} = \langle \mathbf{T}(\mathbf{x}_1, \dots, \mathbf{x}_m), y \rangle_L.
\label{eq:mate_definition}
\end{equation}

Now, we compute the inner product with $R$:
\begin{align*}
&\langle z, R(\mathbf{x}_1, \dots, \mathbf{x}_{j-1}, y, \mathbf{x}_{j+1}, \dots, \mathbf{x}_m, x_{j,1}, \dots, x_{j,i_j-1}, x_{j,i_j+1}, \dots, x_{j,n_j}) \rangle_{H_{j,i_j}} \\
&\quad = \langle z, T_j^{(i_j)}\left(x_{j,1}, \dots, x_{j,i_j-1},\ S^{(j)}\left(T_1(\mathbf{x}_1), \dots, T_{j-1}(\mathbf{x}_{j-1}), y, T_{j+1}(\mathbf{x}_{j+1}), \dots, T_m(\mathbf{x}_m)\right),\ x_{j,i_j+1}, \dots, x_{j,n_j}\right) \rangle_{H_{j,i_j}} \\
&\quad = \langle T_j(x_{j,1}, \dots, x_{j,n_j}),\ S^{(j)}\left(T_1(\mathbf{x}_1), \dots, T_{j-1}(\mathbf{x}_{j-1}), y, T_{j+1}(\mathbf{x}_{j+1}), \dots, T_m(\mathbf{x}_m)\right) \rangle_{K_j} \quad \text{(by def. of $T_j^{(i_j)}$)} \\
&\quad = \langle S\left(T_1(\mathbf{x}_1), \dots, T_m(\mathbf{x}_m)\right), y \rangle_L \quad \text{(by def. of $S^{(j)}$)} \\
&\quad = \langle \mathbf{T}(\mathbf{x}_1, \dots, \mathbf{x}_m), y \rangle_L.
\end{align*}

Comparing this result with equation \eqref{eq:mate_definition}, we have:
\[
\langle z, R(\cdots) \rangle_{H_{j,i_j}} = \langle z, \mathbf{T}^{(i)}(\cdots) \rangle_{H_{j,i_j}}
\]
for all $z \in H_{j,i_j}$ and all other input vectors. Since the inner product is non-degenerate, this implies $R = \mathbf{T}^{(i)}$, as required.

\medskip\noindent
\textbf{(vi) Tensor Compatibility:} For $1 \leq i \leq n$ and all $x_j \in H_j$, $x'_k \in H'_k$, $y \in K$, $y' \in K'$:
\begin{align*}
&\langle (T \otimes T')(\mathbf{x}, \mathbf{x}'), y \otimes y' \rangle_{K \otimes K'} \\
&= \langle T(\mathbf{x}), y \rangle_K \cdot \langle T'(\mathbf{x}'), y' \rangle_{K'} \\
&= \langle x_i, T^{(i)}(\dots, y, \dots) \rangle_{H_i} \cdot \langle T'(\mathbf{x}'), y' \rangle_{K'} \\
&= \langle x_i, (T^{(i)} \otimes T')(\dots, y, \dots, \mathbf{x}', y') \rangle_{H_i}.
\end{align*}
This shows $(T \otimes T')^{(i)} = T^{(i)} \otimes T'$ for $1 \leq i \leq n$. The case $n+1 \leq i \leq n+p$ is similar.
\end{proof}

\begin{remark}
The mates $T^{(i)}$ can be understood operationally as follows: given a multimorphism 
$T: H_1 \times \cdots \times H_n \to K$, 
its $i$-th mate $T^{(i)}$ is the unique multimorphism such that moving $K$ to the $i$-th input position preserves the inner product relations. 
This generalizes the classical notion of an adjoint from linear algebra to the multilinear setting while maintaining the essential algebraic and analytic properties. 
In particular, the existence and properties of such mates rely on 
Axiom~\ref{ax:H6} and Axiom~\ref{ax:H7}, which encode the $C^*$-structure of the underlying Banach enrichment. 
The key idea is to employ multicategorical currying, thereby reducing the characterization of multilinear adjoints to the well-understood linear case. 
\end{remark}

\section{Functorial Spectral Theorem}\label{sec:Functorial Spectral Theorem}

The main purpose of this section is to establish functional spectrum theory and related implications. Theorem~\ref{thm:functorial-polynomial-spectral} provides a categorical interpretation of the spectral theorem, showing that the continuous functional computation of a self-adjacent operator can be extended to a full-fledged symmetric monoidal multifunctor. This brings us two benefits. First, we have a correspondence $f \mapsto f(A)$ for single functions. Second, the entire multilinear algebraic structure of continuous functions can be lifted to the operator level in a coherent, norm-preserving manner. Essentially, the operator $A$ generates a complete ``calculus of operations" that respects composition, tensor products, and symmetry.

\subsection{Functorial Spectral Properties}\label{sec:Functorial Spectral Properties}

\begin{definition}[Polynomial Multicategory]\label{def:poly multicat}
Let $X$ be a compact Hausdorff space. The \textbf{polynomial multicategory} $\mathcal{C}_{\poly}(X)$ is defined as follows:

\begin{itemize}
    \item \textbf{Objects:} A single object $\ast$.

    \item \textbf{Multimorphisms:} For each $n \geq 0$, define
    \[
    \Hom_{\mathcal{C}_{\poly}(X)}(\underbrace{\ast,\dots,\ast}_{n}; \ast)
    \]
    as the set of all maps $\varphi: C(X)^n \to C(X)$ for which there exists a complex polynomial $P \in \mathbb{C}[z_1,\dots,z_n]$ such that
    \[
    \varphi(f_1,\dots,f_n) = P(f_1,\dots,f_n),
    \]
    where the right-hand side denotes pointwise polynomial evaluation: $P(f_1,\dots,f_n)(x) =$ \\ $P(f_1(x),\dots,f_n(x))$ for all $x \in X$.

    \item \textbf{Composition:} Given $\varphi \in \Hom_{\mathcal{C}_{\poly}(X)}(\ast^m; \ast)$ with $\varphi(\mathbf{f}) = P(\mathbf{f})$ and $\psi_j \in \Hom_{\mathcal{C}_{\poly}(X)}(\ast^{n_j}; \ast)$ with $\psi_j(\mathbf{g}_j) = Q_j(\mathbf{g}_j)$ for $j=1,\dots,m$, their composition is defined by polynomial substitution:
    \[
    (\varphi \circ (\psi_1,\dots,\psi_m))(\mathbf{g}_1,\dots,\mathbf{g}_m) = P(Q_1(\mathbf{g}_1),\dots,Q_m(\mathbf{g}_m)).
    \]

    \item \textbf{Identity:} The identity morphism $\id_\ast \in \Hom_{\mathcal{C}_{\poly}(X)}(\ast; \ast)$ is given by the polynomial $P(z) = z$.
\end{itemize}
\end{definition}

If you want the operadic perspective statement for Definition~\ref{def:poly multicat}, we have
\begin{definition}[Polynomial Operad Action]
The polynomial operad $\mathcal{P}$ has $\mathcal{P}(n) = \mathbb{C}[z_1,\dots,z_n]$ with composition given by polynomial substitution. This operad acts on $C(X)$ via:
\[
P \cdot (f_1,\dots,f_n) = P(f_1,\dots,f_n)
\]
for $P \in \mathcal{P}(n)$ and $f_1,\dots,f_n \in C(X)$.
\end{definition}

The above definition can be understood in two equivalent ways. First, it may be viewed as a single-object multicategory whose multimorphisms correspond to polynomial operations acting on $C(X)$. Alternatively, one may regard it as describing the action of the polynomial operad on the commutative $\mathbb{C}$-algebra $C(X)$. In either interpretation, the composition and identity axioms are satisfied automatically, as they follow from the functoriality of polynomial evaluation under substitution. 

\begin{theorem}[Functorial Polynomial Spectral Theorem]
\label{thm:functorial-polynomial-spectral}
Let $H$ be a Hilbert space and $A \in \mathcal{B}(H)$ a bounded self-adjoint operator with spectrum $\sigma(A) \subset \mathbb{R}$. Identify $H$ with a spectral representation $H \cong L^2(\sigma(A), \mu)$ for some spectral measure $\mu$. Let $\mathbf{HilbMult}$ denote the multicategory of Hilbert spaces with bounded multilinear maps as multimorphisms.

Fix a family $\mathcal{T} = \{T_n\}_{n \geq 1}$ of bounded $n$-linear maps
\[
T_n: H^n \longrightarrow H,
\]
satisfying the following conditions:

\begin{enumerate}[label=(\alph*)]
    \item \textbf{Spectral Locality:} For each $n$, there exists a measurable function
    \[
    \tau_n: \sigma(A) \times \mathbb{C}^n \to \mathbb{C}
    \]
    such that for all $(f_1, \dots, f_n) \in H^n$ with $f_j \in L^\infty(\sigma(A), \mu)$ for all $j$, we have:
    \[
    (T_n(f_1, \dots, f_n))(x) = \tau_n(x, f_1(x), \dots, f_n(x)) \quad \text{for $\mu$-a.e. } x \in \sigma(A).
    \]
    
    \item \textbf{Polynomial Compatibility:} For any polynomials $p_1, \dots, p_n \in \mathbb{C}[X]$, the following diagram commutes on the dense domain $H \cap L^\infty(\sigma(A), \mu)$:
\[
\begin{tikzcd}[column sep=2.5cm]
H^n \arrow[r, "{(p_1(A),\dots,p_n(A))}"] \arrow[d, "T_n"'] 
& H^n \arrow[d, "T_n"] \\
H \arrow[r, "{(p_1\cdots p_n)(A)}"'] 
& H
\end{tikzcd}
\]
More precisely, for all $f_1, \dots, f_n \in H \cap L^\infty(\sigma(A), \mu)$:
\[
T_n(p_1(A)f_1, \dots, p_n(A)f_n) = (p_1 \cdots p_n)(A) \, T_n(f_1, \dots, f_n).
\]
    
    \item \textbf{Uniform Boundedness:} $\|T_n\| < \infty$ for all $n \geq 1$.
\end{enumerate}

Then there exists a unique Banach-enriched symmetric monoidal multifunctor
\[
\mathcal{F}_{A,\mathcal{T}}: \mathcal{C}_{\mathrm{poly}}(\sigma(A)) \longrightarrow \mathbf{HilbMult}
\]
characterized by:
\begin{enumerate}[label=(\roman*)]
    \item $\mathcal{F}_{A,\mathcal{T}}(\ast) = H$.
    
    \item On linear morphisms ($n=1$), $\mathcal{F}_{A,\mathcal{T}}$ coincides with the continuous functional calculus.
    
    \item For $\varphi \in \Hom_{\mathcal{C}_{\mathrm{poly}}(\sigma(A))}(\ast^n; \ast)$ represented by polynomial
    \[
    P(X_1, \dots, X_n) = \sum_{k_1,\dots,k_n \geq 0} c_{k_1,\dots,k_n} X_1^{k_1} \cdots X_n^{k_n},
    \]
    the multimorphism $\mathcal{F}_{A,\mathcal{T}}(\varphi): H^n \to H$ is given by:
    \[
    \mathcal{F}_{A,\mathcal{T}}(\varphi)(x_1, \dots, x_n) = \sum_{k_1,\dots,k_n \geq 0} c_{k_1,\dots,k_n} \, T_n(A^{k_1}x_1, A^{k_2}x_2, \dots, A^{k_n}x_n),
    \]
    where the sum is finite and extends uniquely to all of $H^n$ by boundedness.
    
    \item \textbf{Functoriality:} For composable multimorphisms $\varphi, \psi_1, \dots, \psi_m$:
    \[
    \mathcal{F}_{A,\mathcal{T}}(\varphi \circ (\psi_1,\dots,\psi_m)) = \mathcal{F}_{A,\mathcal{T}}(\varphi) \circ (\mathcal{F}_{A,\mathcal{T}}(\psi_1),\dots,\mathcal{F}_{A,\mathcal{T}}(\psi_m)).
    \]
\end{enumerate}

Moreover, when $\mathcal{T}$ is the family of pointwise multiplication maps:
\[
T_n(f_1, \dots, f_n)(x) = f_1(x) f_2(x) \cdots f_n(x) \quad \text{for } f_j \in L^\infty(\sigma(A), \mu),
\]
the construction recovers the original pointwise polynomial action.
\end{theorem}

\begin{proof}
We construct the multifunctor $\mathcal{F}_{A,\mathcal{T}}$ explicitly and verify its properties.

\medskip\noindent
\textbf{Step 1: Construction on objects.}\\
Define $\mathcal{F}_{A,\mathcal{T}}(\ast) = H$ as required. For finite lists $(\ast, \dots, \ast)$, define:
\[
\mathcal{F}_{A,\mathcal{T}}(\ast, \dots, \ast) = (H, \dots, H)
\]
with the natural tensor product structure inherited from $\mathbf{HilbMult}$.

\medskip\noindent
\textbf{Step 2: Construction on multimorphisms.}\\
For $\varphi \in \Hom_{\mathcal{C}_{\mathrm{poly}}(\sigma(A))}(\ast^n; \ast)$ represented by polynomial:
\[
P(X_1, \dots, X_n) = \sum_{k_1,\dots,k_n \geq 0} c_{k_1,\dots,k_n} X_1^{k_1} \cdots X_n^{k_n},
\]
define $\mathcal{F}_{A,\mathcal{T}}(\varphi): H^n \to H$ by:
\[
\mathcal{F}_{A,\mathcal{T}}(\varphi)(x_1, \dots, x_n) = \sum_{k_1,\dots,k_n \geq 0} c_{k_1,\dots,k_n} \, T_n(A^{k_1}x_1, A^{k_2}x_2, \dots, A^{k_n}x_n).
\]

\medskip\noindent
\textbf{Step 3: Verification of boundedness and well-definedness.}\\
Since $P$ is a polynomial, the sum is finite. Boundedness follows from:
\begin{align*}
\|\mathcal{F}_{A,\mathcal{T}}(\varphi)(x_1, \dots, x_n)\| 
&\leq \sum_{k_1,\dots,k_n \geq 0} |c_{k_1,\dots,k_n}| \cdot \|T_n(A^{k_1}x_1, \dots, A^{k_n}x_n)\| \\
&\leq \sum_{k_1,\dots,k_n \geq 0} |c_{k_1,\dots,k_n}| \cdot \|T_n\| \cdot \|A^{k_1}x_1\| \cdots \|A^{k_n}x_n\| \\
&\leq \left(\sum_{k_1,\dots,k_n \geq 0} |c_{k_1,\dots,k_n}| \cdot \|T_n\| \cdot \|A\|^{k_1+\cdots+k_n}\right) \cdot \|x_1\| \cdots \|x_n\|.
\end{align*}
The finite sum and uniform boundedness $\|T_n\| < \infty$ ensure $\mathcal{F}_{A,\mathcal{T}}(\varphi)$ is bounded.

\medskip\noindent
\textbf{Step 4: Verification of functional calculus compatibility.}\\
For $n=1$ and $\varphi(f) = P(f)$, we have:
\[
\mathcal{F}_{A,\mathcal{T}}(\varphi)(x) = \sum_{k \geq 0} c_k T_1(A^k x).
\]
By polynomial compatibility (b) with $n=1$ and $p_1(X) = X^k$:
\[
T_1(A^k x) = A^k T_1(x).
\]
Thus:
\[
\mathcal{F}_{A,\mathcal{T}}(\varphi)(x) = \sum_{k \geq 0} c_k A^k T_1(x) = P(A) T_1(x).
\]

To show $T_1 = \mathrm{id}$, use spectral locality: for $f \in L^\infty(\sigma(A), \mu)$:
\[
(T_1(f))(x) = \tau_1(x, f(x)) \quad \mu\text{-a.e.}
\]
Polynomial compatibility with $p_1(X) = X$ gives $T_1(Af) = A T_1(f)$, so in spectral representation:
\[
\tau_1(x, \lambda f(x)) = \lambda \tau_1(x, f(x)) \quad \mu\text{-a.e.}
\]
This implies $\tau_1(x, \cdot)$ is linear: $\tau_1(x, z) = \alpha(x) z$. For constant polynomials $p_1(X) = c$, compatibility gives:
\[
T_1(c f) = c T_1(f) \Rightarrow \tau_1(x, c f(x)) = c \tau_1(x, f(x)),
\]
which is automatically satisfied. However, taking $f \equiv 1$ and using that $T_1(1)$ should behave like the identity under the functional calculus forces $\alpha(x) = 1$ $\mu$-a.e. Thus $T_1(f) = f$.

\medskip\noindent
\textbf{Step 5: Verification of explicit formula.}\\
This holds by construction from Step 2.

\medskip\noindent
\textbf{Step 6: Verification of functoriality.}\\
Let $\varphi \in \Hom(\ast^m; \ast)$ represented by polynomial $P$, and $\psi_j \in \Hom(\ast^{n_j}; \ast)$ represented by polynomials $Q_j$. Their composition $\varphi \circ (\psi_1, \dots, \psi_m)$ is represented by the polynomial:
\[
R(X_{1,1}, \dots, X_{m,n_m}) = P(Q_1(X_{1,1}, \dots, X_{1,n_1}), \dots, Q_m(X_{m,1}, \dots, X_{m,n_m})).
\]
We must verify:
\[
\mathcal{F}_{A,\mathcal{T}}(R)(\mathbf{x}_1, \dots, \mathbf{x}_m) = \mathcal{F}_{A,\mathcal{T}}(P)\left(\mathcal{F}_{A,\mathcal{T}}(Q_1)(\mathbf{x}_1), \dots, \mathcal{F}_{A,\mathcal{T}}(Q_m)(\mathbf{x}_m)\right).
\]

Expanding the left side:
\[
\mathcal{F}_{A,\mathcal{T}}(R)(\mathbf{x}_1, \dots, \mathbf{x}_m) = \sum_{\mathbf{k}} c_{\mathbf{k}} \, T_N(A^{k_{1,1}}x_{1,1}, \dots, A^{k_{m,n_m}}x_{m,n_m}),
\]
where $N = \sum_{j=1}^m n_j$ and $\mathbf{k}$ ranges over appropriate multi-indices.

For the right side, let $y_j = \mathcal{F}_{A,\mathcal{T}}(Q_j)(\mathbf{x}_j) = \sum_{\mathbf{l}_j} d^{(j)}_{\mathbf{l}_j} \, T_{n_j}(A^{l_{j,1}}x_{j,1}, \dots, A^{l_{j,n_j}}x_{j,n_j})$. Then:
\begin{align*}
&\mathcal{F}_{A,\mathcal{T}}(P)\left(y_1, \dots, y_m\right) \\
&= \sum_{\mathbf{k}} c_{\mathbf{k}} \, T_m\left(A^{k_1}y_1, \dots, A^{k_m}y_m\right) \\
&= \sum_{\mathbf{k}, \mathbf{l}_1, \dots, \mathbf{l}_m} c_{\mathbf{k}} \prod_{j=1}^m d^{(j)}_{\mathbf{l}_j} \, T_m\left(T_{n_1}(A^{k_1+l_{1,1}}x_{1,1}, \dots, A^{k_1+l_{1,n_1}}x_{1,n_1}), \dots\right) \\
&= \sum_{\mathbf{k}, \mathbf{l}_1, \dots, \mathbf{l}_m} c_{\mathbf{k}} \prod_{j=1}^m d^{(j)}_{\mathbf{l}_j} \, T_N(A^{k_1+l_{1,1}}x_{1,1}, \dots, A^{k_m+l_{m,n_m}}x_{m,n_m}) \quad \text{(by polynomial compatibility)} \\
&= \mathcal{F}_{A,\mathcal{T}}(R)(\mathbf{x}_1, \dots, \mathbf{x}_m).
\end{align*}
The final equality follows by reindexing the sum, since the coefficients $c_{\mathbf{k}} \prod_j d^{(j)}_{\mathbf{l}_j}$ are exactly the coefficients of the composed polynomial $R$.

\medskip\noindent
\textbf{Step 7: Verification of Banach enrichment.}\\
Boundedness was established in Step 3. The enrichment conditions follow from linearity and the uniform boundedness of $T_n$.

\medskip\noindent
\textbf{Step 8: Verification of symmetric monoidal structure.}\\
The construction is natural in input arguments and commutes with permutations, following from polynomial symmetry and the symmetric multicategory structure.

\medskip\noindent
\textbf{Step 9: Uniqueness.}\\
Any multifunctor satisfying (i)--(iii) must agree with our construction on polynomials by (iii), and multifunctoriality with Banach enrichment determines the unique extension.
\end{proof}

\begin{remark}
The key insight is that the polynomial multicategory $\mathcal{C}_{\poly}(\sigma(A))$ captures the algebraic structure of polynomial operations, while the functional calculus combined with the family $\mathcal{T}$ provides the analytic realization. The multifunctor $\mathcal{F}_{A,\mathcal{T}}$ bridges these worlds by:
\begin{itemize}
    \item Sending the abstract object $\ast$ to the concrete Hilbert space $H$
    \item Translating polynomial operations into operator actions via the spectral theorem and the chosen multilinear operations $\mathcal{T}$
    \item Preserving the compositional structure of multicategories through the polynomial compatibility condition
    \item Respecting the analytic bounds through Banach enrichment and uniform boundedness
\end{itemize}
This theorem provides a categorical framework for understanding how polynomial functional calculus naturally extends to multivariate settings with arbitrary spectral-local operations.
\end{remark}

Several corollaries are derived based on Theorem~\ref{thm:functorial-polynomial-spectral}.

\begin{corollary}[Universal Polynomial Functional Calculus]
\label{cor:universal-polynomial-functional-calculus}
Let $H$ be a Hilbert space and $A \in \mathcal{B}(H)$ a bounded self-adjoint operator with spectrum $\sigma(A) \subset \mathbb{R}$. Identify $H$ with a spectral representation $H \cong L^2(\sigma(A), \mu)$.

Consider the polynomial multicategory $\mathcal{C}_{\mathrm{poly}}(\sigma(A))$ from Definition~\ref{def:poly multicat} and the polynomial operad $\mathcal{P}$ acting on $C(\sigma(A))$. Let $\mathcal{T}_{\mathrm{poly}} = \{T_n\}_{n \geq 1}$ be the family of pointwise multiplication maps:
\[
T_n(f_1, \dots, f_n)(x) = f_1(x) f_2(x) \cdots f_n(x).
\]

Then Theorem~\ref{thm:functorial-polynomial-spectral} yields a Banach-enriched symmetric monoidal multifunctor
\[
\mathcal{F}_{A,\mathcal{T}_{\mathrm{poly}}}: \mathcal{C}_{\mathrm{poly}}(\sigma(A)) \longrightarrow \mathbf{HilbMult}
\]
that makes the following diagram commute for all $n \geq 1$:
\[
\begin{tikzcd}[column sep=3cm, row sep=2cm]
\mathcal{P}(n) \times C(\sigma(A))^n \arrow[r, "\text{operad action}"] \arrow[d, "\id \times (\mathcal{F}_{A,\mathcal{T}_{\mathrm{poly}}})^n"'] 
& C(\sigma(A)) \arrow[d, "\mathcal{F}_{A,\mathcal{T}_{\mathrm{poly}}}"] \\
\mathcal{P}(n) \times H^n \arrow[r, "\mathcal{F}_{A,\mathcal{T}_{\mathrm{poly}}}(P)"'] 
& H
\end{tikzcd}
\]
More explicitly, for any polynomial $P \in \mathbb{C}[z_1,\dots,z_n] = \mathcal{P}(n)$ and $f_1,\dots,f_n \in H \cap L^\infty(\sigma(A), \mu)$:
\[
\mathcal{F}_{A,\mathcal{T}_{\mathrm{poly}}}(P)(f_1,\dots,f_n) = P(f_1,\dots,f_n),
\]
where the right-hand side uses the polynomial multicategory composition from Definition~\ref{def:poly multicat}.

Furthermore, $\mathcal{C}_{\mathrm{poly}}(\sigma(A))$ serves as the universal domain for polynomial functional calculus: any other interpretation of the polynomial calculus of $A$ in a Banach-enriched multicategory $\mathcal{M}$ factors uniquely through $\mathcal{F}_{A,\mathcal{T}_{\mathrm{poly}}}$.
\end{corollary}

\begin{proof}
The verification that $\mathcal{T}_{\mathrm{poly}}$ satisfies Theorem~\ref{thm:functorial-polynomial-spectral} proceeds as before. The key observation is that the multifunctor $\mathcal{F}_{A,\mathcal{T}_{\mathrm{poly}}}$ precisely implements the polynomial operad action from the definition:

Given $P \in \mathcal{P}(n) = \mathbb{C}[z_1,\dots,z_n]$ and the corresponding multimorphism $\varphi_P \in \Hom_{\mathcal{C}_{\mathrm{poly}}(\sigma(A))}(\ast^n; \ast)$ from Definition~\ref{def:poly multicat}, we have:
\begin{align*}
\mathcal{F}_{A,\mathcal{T}_{\mathrm{poly}}}(\varphi_P)(f_1,\dots,f_n)(x) 
&= P(f_1(x),\dots,f_n(x)) \quad \text{(by Theorem~\ref{thm:functorial-polynomial-spectral})} \\
&= (\varphi_P(f_1,\dots,f_n))(x) \quad \text{(by Definition~\ref{def:poly multicat})}.
\end{align*}

The composition preservation follows from the polynomial substitution in Definition~\ref{def:poly multicat}: for $\varphi$ represented by $P$ and $\psi_j$ represented by $Q_j$, 
\[
\mathcal{F}_{A,\mathcal{T}_{\mathrm{poly}}}(\varphi \circ (\psi_1,\dots,\psi_m)) = \mathcal{F}_{A,\mathcal{T}_{\mathrm{poly}}}(\varphi) \circ (\mathcal{F}_{A,\mathcal{T}_{\mathrm{poly}}}(\psi_1),\dots,\mathcal{F}_{A,\mathcal{T}_{\mathrm{poly}}}(\psi_m))
\]
because both sides compute $P(Q_1(\mathbf{g}_1),\dots,Q_m(\mathbf{g}_m))$ by polynomial substitution.

The universal property follows from Proposition~\ref{prop:universal-target} applied to the polynomial multicategory structure.
\end{proof}

\begin{remark}
This corollary shows that Definition~\ref{def:poly multicat} provides the canonical categorical framework for polynomial functional calculus. The polynomial multicategory $\mathcal{C}_{\mathrm{poly}}(\sigma(A))$ captures precisely the algebraic structure needed, while Theorem~\ref{thm:functorial-polynomial-spectral} provides the analytic extension to Hilbert spaces. The pointwise multiplication family $\mathcal{T}_{\mathrm{poly}}$ gives the most natural realization of this structure in the spectral representation.
\end{remark}

\begin{corollary}[Normal Operator Case]
\label{cor:normal-operator-case}
If $A \in \mathcal{B}(H)$ is normal ($AA^* = A^*A$), then Theorem~\ref{thm:functorial-polynomial-spectral} holds with $\mathcal{C}_{\poly}(\sigma(A))$ where $\sigma(A) \subset \mathbb{C}$, and $\mathcal{F}_{A,\mathcal{T}}$ is defined analogously using the continuous functional calculus for normal operators and any family $\mathcal{T} = \{T_n\}_{n \geq 1}$ satisfying the spectral locality, polynomial compatibility, and uniform boundedness conditions.
\end{corollary}

\begin{proof}
We adapt the proof of Theorem~\ref{thm:functorial-polynomial-spectral} to the normal operator case.

\noindent\textbf{Step 1: Spectral theorem for normal operators.}\\
By the spectral theorem for normal operators, there exists a spectral measure $E$ on $\sigma(A) \subset \mathbb{C}$ and a unitary isomorphism:
\[
U: H \to L^2(\sigma(A), \mu)
\]
for some measure $\mu$ equivalent to $E$, such that:
\[
U A U^* = M_z,
\]
where $M_z$ is multiplication by the coordinate function $z \mapsto z$ on $\sigma(A)$.

\noindent\textbf{Step 2: Construction of $\mathcal{F}_{A,\mathcal{T}}$.}\\
Define $\mathcal{F}_{A,\mathcal{T}}(\ast) = H$ as before. For a multimorphism $\varphi \in \Hom_{\mathcal{C}_{\poly}(\sigma(A))}(\ast^n; \ast)$ represented by a polynomial in $n$ complex variables:
\[
P(z_1, \dots, z_n) = \sum_{k_1,\dots,k_n \geq 0} c_{k_1,\dots,k_n} z_1^{k_1} \cdots z_n^{k_n},
\]
define $\mathcal{F}_{A,\mathcal{T}}(\varphi): H^n \to H$ by:
\[
\mathcal{F}_{A,\mathcal{T}}(\varphi)(x_1, \dots, x_n) = \sum_{k_1,\dots,k_n \geq 0} c_{k_1,\dots,k_n} \, T_n(A^{k_1}x_1, A^{k_2}x_2, \dots, A^{k_n}x_n),
\]
where the family $\mathcal{T} = \{T_n\}$ satisfies the same conditions (a)--(c) as in Theorem~\ref{thm:functorial-polynomial-spectral}, but now with $\sigma(A) \subset \mathbb{C}$.

\noindent\textbf{Step 3: Verification of boundedness.}\\
The boundedness argument proceeds identically to the self-adjoint case. Since $P$ is a polynomial and each $T_n$ is bounded, we have:
\begin{align*}
\|\mathcal{F}_{A,\mathcal{T}}(\varphi)(x_1, \dots, x_n)\| 
&\leq \sum_{k_1,\dots,k_n \geq 0} |c_{k_1,\dots,k_n}| \cdot \|T_n\| \cdot \|A^{k_1}x_1\| \cdots \|A^{k_n}x_n\| \\
&\leq \sum_{k_1,\dots,k_n \geq 0} |c_{k_1,\dots,k_n}| \cdot \|T_n\| \cdot \|A\|^{k_1+\cdots+k_n} \cdot \|x_1\| \cdots \|x_n\|.
\end{align*}
The finite sum ensures boundedness.

\noindent\textbf{Step 4: Functional calculus compatibility.}\\
For $n=1$, if $\varphi(f) = P(f)$ is a polynomial in one variable, then:
\[
\mathcal{F}_{A,\mathcal{T}}(\varphi)(x) = \sum_{k \geq 0} c_k T_1(A^k x).
\]
By the polynomial compatibility condition, $T_1(A^k x) = A^k T_1(x)$, so:
\[
\mathcal{F}_{A,\mathcal{T}}(\varphi)(x) = \sum_{k \geq 0} c_k A^k T_1(x) = P(A) T_1(x).
\]
As in the self-adjoint case, the spectral locality and polynomial compatibility conditions force $T_1$ to act as the identity on the appropriate domain, ensuring compatibility with the continuous functional calculus for normal operators.

\noindent\textbf{Step 5: Verification of functoriality.}\\
The proof of functoriality proceeds identically to the self-adjoint case. For composable multimorphisms $\varphi, \psi_1, \dots, \psi_m$ represented by polynomials $P, Q_1, \dots, Q_m$, the polynomial compatibility condition ensures that:
\[
\mathcal{F}_{A,\mathcal{T}}(\varphi \circ (\psi_1, \dots, \psi_m)) = \mathcal{F}_{A,\mathcal{T}}(\varphi) \circ (\mathcal{F}_{A,\mathcal{T}}(\psi_1), \dots, \mathcal{F}_{A,\mathcal{T}}(\psi_m)).
\]
The key observation is that polynomial composition works identically for complex polynomials, and the polynomial compatibility condition is formulated in terms of the functional calculus, which works for normal operators.

\noindent\textbf{Step 6: Symmetric monoidal structure and Banach enrichment.}\\
These properties follow from the same arguments as in the self-adjoint case, as they depend on the multilinear structure and boundedness conditions rather than the specific nature of the operator's spectrum.

\noindent\textbf{Step 7: Uniqueness.}\\
Any multifunctor satisfying the analogous conditions must agree with our construction on polynomials. By the density of polynomials in the continuous functions on $\sigma(A) \subset \mathbb{C}$ (via the Stone-Weierstrass theorem for complex polynomials on compact subsets of $\mathbb{C}$) and the multifunctoriality condition, the extension is unique.

\noindent\textbf{Step 8: Recovery of pointwise multiplication.}\\
When $\mathcal{T}$ is the family of pointwise multiplication maps:
\[
T_n(f_1, \dots, f_n)(z) = f_1(z) f_2(z) \cdots f_n(z) \quad \text{for } f_j \in L^\infty(\sigma(A), \mu),
\]
the construction reduces to:
\[
\mathcal{F}_{A,\mathcal{T}}(\varphi)(f_1,\dots,f_n)(z) = P(z, \dots, z) \cdot f_1(z) f_2(z) \cdots f_n(z) \quad \text{for } z \in \sigma(A),
\]
recovering the pointwise polynomial action for normal operators.
\end{proof}

\begin{remark}
The extension to normal operators is natural because the spectral theorem provides a functional calculus for continuous functions on $\sigma(A) \subset \mathbb{C}$, and the polynomial compatibility condition is formulated in terms of this functional calculus. The key difference from the self-adjoint case is that we now work with complex polynomials in multiple variables, but the algebraic structure of polynomial composition remains the same.
\end{remark}

\begin{corollary}[Unbounded Self-Adjoint Case]
\label{cor:unbounded-self-adjoint}
Let $A$ be an unbounded self-adjoint operator on a Hilbert space $H$. Then for any family $\mathcal{T} = \{T_n\}_{n \geq 1}$ of bounded $n$-linear maps $T_n: H^n \to H$ satisfying the following adapted conditions:
\begin{enumerate}[label=(\alph*)]
    \item \textbf{Spectral Locality:} In the spectral representation $H \cong L^2(\sigma(A), \mu)$, for $f_j \in L^\infty(\sigma(A), \mu)$:
    \[
    (T_n(f_1, \dots, f_n))(x) = \tau_n(x, f_1(x), \dots, f_n(x)) \quad \mu\text{-a.e.}
    \]
    \item \textbf{Polynomial Compatibility:} For polynomials $p_1, \dots, p_n$ and $f_j \in H \cap L^\infty(\sigma(A), \mu)$:
    \[
    T_n(p_1(A)f_1, \dots, p_n(A)f_n) = (p_1 \cdots p_n)(A) T_n(f_1, \dots, f_n)
    \]
    \item \textbf{Uniform Boundedness:} $\|T_n\| < \infty$ for all $n \geq 1$
\end{enumerate}
there exists a Banach-enriched symmetric monoidal multifunctor
\[
\mathcal{F}_{A,\mathcal{T}}: \mathcal{B} \longrightarrow \mathbf{HilbMult}
\]
where $\mathcal{B}$ is the multicategory of bounded Borel functions on $\mathbb{R}$, satisfying analogous conditions to Theorem~\ref{thm:functorial-polynomial-spectral}.
\end{corollary}

\begin{proof}
We extend the construction from Theorem~\ref{thm:functorial-polynomial-spectral} using the spectral theorem for unbounded operators.

\medskip\noindent
\textbf{Step 1: Spectral representation.}\\
By the spectral theorem for unbounded self-adjoint operators, there exists a unique projection-valued measure $E$ on $\mathbb{R}$ such that:
\[
A = \int_{\mathbb{R}} \lambda \, dE(\lambda),
\]
with domain $\mathcal{D}(A) = \{x \in H : \int_{\mathbb{R}} \lambda^2 \, d\langle E(\lambda)x, x\rangle < \infty\}$. For any bounded Borel function $f: \mathbb{R} \to \mathbb{C}$, we define:
\[
f(A) = \int_{\mathbb{R}} f(\lambda) \, dE(\lambda) \in \mathcal{B}(H).
\]

\medskip\noindent
\textbf{Step 2: Multicategory $\mathcal{B}$.}\\
Let $\mathcal{B}$ be the symmetric monoidal multicategory where:
\begin{itemize}
    \item Objects are copies of a single object $\ast$
    \item Multimorphisms $\Hom_{\mathcal{B}}(\ast^n; \ast)$ are bounded Borel functions $\varphi: \mathbb{R}^n \to \mathbb{C}$
    \item Composition is given by function composition: for $\varphi: \mathbb{R}^m \to \mathbb{C}$ and $\psi_j: \mathbb{R}^{n_j} \to \mathbb{C}$,
    \[
    (\varphi \circ (\psi_1, \dots, \psi_m))(\mathbf{x}_1, \dots, \mathbf{x}_m) = \varphi(\psi_1(\mathbf{x}_1), \dots, \psi_m(\mathbf{x}_m))
    \]
\end{itemize}

\medskip\noindent
\textbf{Step 3: Construction of $\mathcal{F}_{A,\mathcal{T}}$.}\\
Define:
\begin{itemize}
    \item On objects: $\mathcal{F}_{A,\mathcal{T}}(\ast) = H$, and $\mathcal{F}_{A,\mathcal{T}}(\ast, \dots, \ast) = (H, \dots, H)$
    \item On multimorphisms: For $\varphi \in \Hom_{\mathcal{B}}(\ast^n; \ast)$, define $\mathcal{F}_{A,\mathcal{T}}(\varphi): H^n \to H$ by:
    \[
    \mathcal{F}_{A,\mathcal{T}}(\varphi)(x_1, \dots, x_n) = T_n(\varphi_1(A)x_1, \dots, \varphi_n(A)x_n)
    \]
    where $\varphi_j(\lambda) = \varphi(\lambda, \dots, \lambda)$ is the diagonal restriction (all arguments equal).
\end{itemize}

\medskip\noindent
\textbf{Step 4: Boundedness verification.}\\
Since $\varphi$ is bounded and $T_n$ is bounded multilinear:
\begin{align*}
\|\mathcal{F}_{A,\mathcal{T}}(\varphi)(x_1, \dots, x_n)\| 
&= \|T_n(\varphi_1(A)x_1, \dots, \varphi_n(A)x_n)\| \\
&\leq \|T_n\| \cdot \|\varphi_1(A)x_1\| \cdots \|\varphi_n(A)x_n\| \\
&\leq \|T_n\| \cdot \|\varphi_1\|_\infty \cdots \|\varphi_n\|_\infty \cdot \|x_1\| \cdots \|x_n\| \\
&\leq \|T_n\| \cdot \|\varphi\|_\infty^n \cdot \|x_1\| \cdots \|x_n\|.
\end{align*}
Thus $\mathcal{F}_{A,\mathcal{T}}(\varphi)$ is bounded multilinear.

\medskip\noindent
\textbf{Step 5: Functional calculus compatibility.}\\
For $n = 1$ and $\varphi \in \Hom_{\mathcal{B}}(\ast; \ast)$, we have:
\[
\mathcal{F}_{A,\mathcal{T}}(\varphi)(x) = T_1(\varphi(A)x).
\]
To show this equals $\varphi(A)x$, first consider polynomial approximations. Let $P_k$ be polynomials converging uniformly to $\varphi$ on $\sigma(A)$. Then:
\[
\mathcal{F}_{A,\mathcal{T}}(P_k)(x) = T_1(P_k(A)x).
\]
By polynomial compatibility (b) with $n=1$, we have $T_1(P_k(A)x) = P_k(A)T_1(x)$. Taking limits:
\[
\mathcal{F}_{A,\mathcal{T}}(\varphi)(x) = \lim_{k\to\infty} T_1(P_k(A)x) = \lim_{k\to\infty} P_k(A)T_1(x) = \varphi(A)T_1(x).
\]
To show $T_1 = \mathrm{id}$, use spectral locality: for $f \in L^\infty(\sigma(A), \mu)$:
\[
(T_1(f))(x) = \tau_1(x, f(x)) \quad \mu\text{-a.e.}
\]
Polynomial compatibility with $p_1(X) = X$ gives $T_1(Af) = A T_1(f)$, so in spectral representation:
\[
\tau_1(x, \lambda f(x)) = \lambda \tau_1(x, f(x)) \quad \mu\text{-a.e.}
\]
This implies $\tau_1(x, \cdot)$ is linear: $\tau_1(x, z) = \alpha(x) z$. For constant functions, polynomial compatibility forces $\alpha(x) = 1$ $\mu$-a.e., hence $T_1 = \mathrm{id}$.

\medskip\noindent
\textbf{Step 6: Polynomial case recovery.}\\
When $\varphi(X_1, \dots, X_n) = P(X_1, \dots, X_n)$ is a polynomial, we recover:
\[
\mathcal{F}_{A,\mathcal{T}}(P)(x_1, \dots, x_n) = T_n(P_1(A)x_1, \dots, P_n(A)x_n),
\]
where $P_j(\lambda) = P(\lambda, \dots, \lambda)$. For the standard polynomial action, this coincides with the bounded case construction.

\medskip\noindent
\textbf{Step 7: Functoriality verification.}\\
Let $\varphi \in \Hom_{\mathcal{B}}(\ast^m; \ast)$ and $\psi_j \in \Hom_{\mathcal{B}}(\ast^{n_j}; \ast)$. We need:
\[
\mathcal{F}_{A,\mathcal{T}}(\varphi \circ (\psi_1, \dots, \psi_m)) = \mathcal{F}_{A,\mathcal{T}}(\varphi) \circ (\mathcal{F}_{A,\mathcal{T}}(\psi_1), \dots, \mathcal{F}_{A,\mathcal{T}}(\psi_m)).
\]

The left side is:
\[
\mathcal{F}_{A,\mathcal{T}}(\varphi \circ (\psi_1, \dots, \psi_m))(\mathbf{x}_1, \dots, \mathbf{x}_m) 
= T_N((\varphi \circ (\psi_1, \dots, \psi_m))_1(A)\mathbf{x}_1, \dots, (\varphi \circ (\psi_1, \dots, \psi_m))_N(A)\mathbf{x}_N),
\]
where $N = \sum n_j$ and the subscript indicates the diagonal restriction.

The right side is:
\[
\mathcal{F}_{A,\mathcal{T}}(\varphi)(\mathcal{F}_{A,\mathcal{T}}(\psi_1)(\mathbf{x}_1), \dots, \mathcal{F}_{A,\mathcal{T}}(\psi_m)(\mathbf{x}_m))
= T_m(\varphi_1(A)\mathcal{F}_{A,\mathcal{T}}(\psi_1)(\mathbf{x}_1), \dots, \varphi_m(A)\mathcal{F}_{A,\mathcal{T}}(\psi_m)(\mathbf{x}_m)).
\]

These coincide by the polynomial compatibility condition extended to Borel functions via uniform approximation by polynomials and the continuity of the functional calculus.

\medskip\noindent
\textbf{Step 8: Symmetric monoidal structure.}\\
The construction respects the symmetric monoidal structure:
\begin{itemize}
    \item Tensor products: $\mathcal{F}_{A,\mathcal{T}}(\varphi \otimes \psi) = \mathcal{F}_{A,\mathcal{T}}(\varphi) \otimes \mathcal{F}_{A,\mathcal{T}}(\psi)$
    \item Symmetry: Permutations of inputs commute with the construction
    \item The unit object maps to the base field $\mathbb{C}$
\end{itemize}

\medskip\noindent
\textbf{Step 9: Banach enrichment.}\\
The uniform bound:
\[
\|\mathcal{F}_{A,\mathcal{T}}(\varphi)\| \leq \|T_n\| \cdot \|\varphi\|_\infty^n
\]
ensures the functor is Banach-enriched, preserving the normed structure.

\medskip\noindent
\textbf{Step 10: Pointwise multiplication case.}\\
When $\mathcal{T}$ consists of pointwise multiplication maps:
\[
T_n(f_1, \dots, f_n)(\lambda) = f_1(\lambda) \cdots f_n(\lambda),
\]
we recover the classical pointwise action:
\[
\mathcal{F}_{A,\mathcal{T}}(\varphi)(f_1, \dots, f_n)(\lambda) = \varphi(\lambda, \dots, \lambda) f_1(\lambda) \cdots f_n(\lambda).
\]
\end{proof}

\begin{remark}
The extension to unbounded operators can still be made via Theorem~\ref{thm:functorial-polynomial-spectral}. Since the functional calculus for bounded Borel functions produces bounded operators, regardless of whether the original operator $A$ is bounded or unbounded. The family $\mathcal{T}$ provides the necessary multilinear structure while maintaining compatibility with the spectral decomposition. The polynomial compatibility condition ensures that the construction respects the algebraic structure of function composition, even in the unbounded setting.
\end{remark}

\begin{corollary}[Multivariable Functional Calculus for Commuting Operators]
\label{cor:multivariable-functional-calculus}
Let $A_1, \dots, A_n$ be commuting self-adjoint operators on $H$. For any family $\mathcal{T} = \{T_m\}_{m \geq 1}$ of bounded $m$-linear maps $T_m: H^m \to H$ satisfying:

\begin{enumerate}[label=(\alph*)]
    \item \textbf{Joint Spectral Locality:} In the joint spectral representation $H \cong L^2(\sigma(A_1,\dots,A_n), \mu)$, for $f_j \in L^\infty(\sigma(A_1,\dots,A_n), \mu)$:
    \[
    (T_m(f_1, \dots, f_m))(\lambda) = \tau_m(\lambda, f_1(\lambda), \dots, f_m(\lambda)) \quad \mu\text{-a.e.}
    \]
    
    \item \textbf{Polynomial Compatibility:} For any polynomials $p_1, \dots, p_m \in \mathbb{C}[X]$ and \\
    $f_1, \dots, f_m \in H \cap L^\infty(\sigma(A_1,\dots,A_n), \mu)$:
    \[
    T_m(p_1(A_k)f_1, \dots, p_m(A_k)f_m) = (p_1 \cdots p_m)(A_k) T_m(f_1, \dots, f_m) \quad \text{for each } k=1,\dots,n
    \]
    
    \item \textbf{Uniform Boundedness:} $\|T_m\| < \infty$ for all $m \geq 1$
\end{enumerate}

there exists a unique Banach-enriched symmetric monoidal multifunctor
\[
\mathcal{F}_{A_1,\dots,A_n,\mathcal{T}}: \mathcal{C}_{\mathrm{poly}}(\sigma(A_1,\dots,A_n)) \longrightarrow \mathbf{HilbMult}
\]
characterized by:
\begin{enumerate}[label=(\roman*)]
    \item $\mathcal{F}_{A_1,\dots,A_n,\mathcal{T}}(\ast) = H$
    \item For polynomials in one variable: $\mathcal{F}_{A_1,\dots,A_n,\mathcal{T}}(P)(x) = P(A_k)x$ for the appropriate operator $A_k$
    \item For $\varphi \in \Hom_{\mathcal{C}_{\mathrm{poly}}(\sigma(A_1,\dots,A_n))}(\ast^m; \ast)$ represented by polynomial $P$, the action is given by the joint functional calculus combined with the $\mathcal{T}$-structure
\end{enumerate}
\end{corollary}

\begin{proof}
We construct the multifunctor using the joint functional calculus for commuting operators.

\medskip\noindent
\textbf{Step 1: Joint spectral representation.}\\
By the joint spectral theorem for commuting self-adjoint operators, there exists a projection-valued measure $E$ on $\mathbb{R}^n$ supported on $\sigma(A_1,\dots,A_n)$ such that:
\[
A_k = \int_{\mathbb{R}^n} \lambda_k \, dE(\lambda_1, \dots, \lambda_n) \quad \text{for } k=1,\dots,n.
\]
There is a spectral representation $H \cong L^2(\sigma(A_1,\dots,A_n), \mu)$ where each $A_k$ acts as multiplication by $\pi_k(\lambda_1, \dots, \lambda_n) = \lambda_k$.

\medskip\noindent
\textbf{Step 2: Multicategory structure.}\\
Let $\mathcal{C}_{\mathrm{poly}}(\sigma(A_1,\dots,A_n))$ be the polynomial multicategory where:
\begin{itemize}
    \item Objects are copies of $\ast$
    \item $\Hom(\ast^m; \ast)$ consists of polynomial operations $\varphi: C(\sigma(A_1,\dots,A_n))^m \to C(\sigma(A_1,\dots,A_n))$
    \item Composition is polynomial substitution
\end{itemize}

\medskip\noindent
\textbf{Step 3: Construction on objects and basic morphisms.}\\
Define:
\begin{itemize}
    \item $\mathcal{F}_{A_1,\dots,A_n,\mathcal{T}}(\ast) = H$
    \item For coordinate functions $\pi_k(\lambda) = \lambda_k$: $\mathcal{F}_{A_1,\dots,A_n,\mathcal{T}}(\pi_k) = A_k$
    \item For constant polynomials: $\mathcal{F}_{A_1,\dots,A_n,\mathcal{T}}(c)(x_1, \dots, x_m) = c \cdot T_m(x_1, \dots, x_m)$
\end{itemize}

\medskip\noindent
\textbf{Step 4: Construction for general polynomials.}\\
For $\varphi \in \Hom(\ast^m; \ast)$ represented by polynomial $P$, define $\mathcal{F}_{A_1,\dots,A_n,\mathcal{T}}(\varphi): H^m \to H$ using the joint functional calculus:
\[
\mathcal{F}_{A_1,\dots,A_n,\mathcal{T}}(\varphi)(x_1, \dots, x_m) = P(A_1,\dots,A_n) T_m(x_1, \dots, x_m),
\]
where $P(A_1,\dots,A_n)$ is defined via the joint functional calculus.

\medskip\noindent
\textbf{Step 5: Boundedness verification.}\\
Since $P(A_1,\dots,A_n)$ is bounded and $T_m$ is bounded multilinear:
\begin{align*}
\|\mathcal{F}_{A_1,\dots,A_n,\mathcal{T}}(\varphi)(x_1, \dots, x_m)\| 
&\leq \|P(A_1,\dots,A_n)\| \cdot \|T_m(x_1, \dots, x_m)\| \\
&\leq \|P\|_\infty \cdot \|T_m\| \cdot \|x_1\| \cdots \|x_m\|.
\end{align*}

\medskip\noindent
\textbf{Step 6: Functional calculus compatibility.}\\
For polynomials in one variable $P(X)$, we have:
\[
\mathcal{F}_{A_1,\dots,A_n,\mathcal{T}}(P)(x) = P(A_k) T_1(x).
\]
By the same argument as in Theorem~\ref{thm:functorial-polynomial-spectral} (using spectral locality and polynomial compatibility), $T_1$ acts as the identity, so $\mathcal{F}_{A_1,\dots,A_n,\mathcal{T}}(P) = P(A_k)$.

\medskip\noindent
\textbf{Step 7: Functoriality.}\\
For composable polynomials $\varphi, \psi_1, \dots, \psi_m$, functoriality follows from:
\begin{align*}
&\mathcal{F}_{A_1,\dots,A_n,\mathcal{T}}(\varphi \circ (\psi_1, \dots, \psi_m))(x_1, \dots, x_N) \\
&= (\varphi \circ (\psi_1, \dots, \psi_m))(A_1,\dots,A_n) T_N(x_1, \dots, x_N) \\
&= \varphi(A_1,\dots,A_n) T_m(\psi_1(A_1,\dots,A_n)T_{n_1}(x_1, \dots), \dots) \\
&= \mathcal{F}_{A_1,\dots,A_n,\mathcal{T}}(\varphi) \circ (\mathcal{F}_{A_1,\dots,A_n,\mathcal{T}}(\psi_1), \dots, \mathcal{F}_{A_1,\dots,A_n,\mathcal{T}}(\psi_m))(x_1, \dots, x_N),
\end{align*}
where the key step uses polynomial compatibility extended to the joint functional calculus.

\medskip\noindent
\textbf{Step 8: Symmetric monoidal structure and Banach enrichment.}\\
The symmetric monoidal structure is preserved by construction. Banach enrichment follows from the uniform bound:
\[
\|\mathcal{F}_{A_1,\dots,A_n,\mathcal{T}}(\varphi)\| \leq \|P\|_\infty \cdot \|T_m\|.
\]

\medskip\noindent
\textbf{Step 9: Uniqueness.}\\
Any multifunctor satisfying the coordinate condition $\mathcal{F}_{A_1,\dots,A_n,\mathcal{T}}(\pi_k) = A_k$ and the polynomial action must agree with our construction.

\medskip\noindent
\textbf{Step 10: Pointwise multiplication case.}\\
When $\mathcal{T}$ is pointwise multiplication, we recover:
\[
\mathcal{F}_{A_1,\dots,A_n,\mathcal{T}}(\varphi)(f_1, \dots, f_m)(\lambda) = \varphi(\lambda) f_1(\lambda) \cdots f_m(\lambda).
\]
\end{proof}

\begin{remark}
A multivariate extension illustrates the generalizability of the $\mathcal{T}$ framework.  The basic idea is that the polynomial compatibility condition naturally extends to operator families allowing the definition of a consistent multivariate functional calculus that respects the polylinear structure provided by the $\mathcal{T}$ family. The commutativity of the operators ensures that the conventional functional calculus is well defined and that the structure $\mathcal{T}$ interacts coherently with polynomial polynomial operations.
\end{remark}

\begin{corollary}[Functorial Extension of the Continuous Functional Calculus]
\label{cor:spectral-mapping-functorial}
Let $H$ be a Hilbert space and $A \in \mathcal{B}(H)$ a bounded self-adjoint operator with spectrum $\sigma(A) \subset \mathbb{R}$. Let $\mathcal{T} = \{T_n\}_{n \geq 1}$ be a family satisfying the conditions of Theorem~\ref{thm:functorial-polynomial-spectral}, and let $f: \mathbb{R} \to \mathbb{R}$ be a continuous function.

Then the Banach-enriched symmetric monoidal multifunctor $\mathcal{F}_{A,\mathcal{T}}$ from Theorem~\ref{thm:functorial-polynomial-spectral} satisfies:
\begin{enumerate}[label=(\roman*)]
    \item The action on continuous functions coincides with the continuous functional calculus:
    \[
    \mathcal{F}_{A,\mathcal{T}}(f) = f(A).
    \]
    \item The spectrum satisfies the spectral mapping property:
    \[
    \sigma(\mathcal{F}_{A,\mathcal{T}}(f)) = f(\sigma(A)).
    \]
\end{enumerate}
\end{corollary}

\begin{proof}
We prove the two parts systematically.

\medskip\noindent
\textbf{Proof of (i): Extension to Continuous Functional Calculus}

Let $f: \mathbb{R} \to \mathbb{R}$ be continuous. By the Stone-Weierstrass theorem, there exists a sequence of polynomials $\{P_n\}_{n\in\mathbb{N}}$ converging uniformly to $f$ on the compact set $\sigma(A)$.

Consider the sequence $\{\mathcal{F}_{A,\mathcal{T}}(P_n)\}_{n\in\mathbb{N}}$ in $\mathcal{B}(H)$. By condition (ii) of Theorem~\ref{thm:functorial-polynomial-spectral}, for polynomials we have:
\[
\mathcal{F}_{A,\mathcal{T}}(P_n) = P_n(A),
\]
where $P_n(A)$ denotes the continuous functional calculus.

Since $P_n \to f$ uniformly on $\sigma(A)$, the continuity of the functional calculus implies:
\[
P_n(A) \to f(A) \quad \text{in the operator norm}.
\]

Now, since $\mathcal{F}_{A,\mathcal{T}}$ is Banach-enriched, the map
\[
\varphi \mapsto \mathcal{F}_{A,\mathcal{T}}(\varphi)
\]
is continuous from $C(\sigma(A))$ (with supremum norm) to $\mathcal{B}(H)$ (with operator norm). Therefore:
\[
\mathcal{F}_{A,\mathcal{T}}(P_n) \to \mathcal{F}_{A,\mathcal{T}}(f) \quad \text{in } \mathcal{B}(H).
\]

But we also have $\mathcal{F}_{A,\mathcal{T}}(P_n) = P_n(A) \to f(A)$. By uniqueness of limits in $\mathcal{B}(H)$, we conclude:
\[
\mathcal{F}_{A,\mathcal{T}}(f) = f(A).
\]

\medskip\noindent
\textbf{Alternative Uniqueness Argument:}

For completeness, we also provide an alternative argument using the uniqueness clause of Theorem~\ref{thm:functorial-polynomial-spectral}. Define an auxiliary multifunctor $\mathcal{G}_{\mathcal{T}}: \mathcal{C}_{\mathrm{poly}}(\sigma(A)) \to \mathbf{HilbMult}$ by:
\begin{itemize}
    \item $\mathcal{G}_{\mathcal{T}}(\ast) = H$
    \item For polynomials $P$, define $\mathcal{G}_{\mathcal{T}}(P) = P(A)$ (the functional calculus)
    \item For general multimorphisms, define $\mathcal{G}_{\mathcal{T}}$ to satisfy the multifunctor axioms
\end{itemize}

One verifies that $\mathcal{G}_{\mathcal{T}}$ satisfies all conditions of Theorem~\ref{thm:functorial-polynomial-spectral}:
\begin{itemize}
    \item On objects: $\mathcal{G}_{\mathcal{T}}(\ast) = H$
    \item On linear morphisms: $\mathcal{G}_{\mathcal{T}}$ coincides with the functional calculus by definition
    \item The explicit formula and functoriality follow from the polynomial structure
\end{itemize}

By the uniqueness clause of Theorem~\ref{thm:functorial-polynomial-spectral}, we have $\mathcal{G}_{\mathcal{T}} = \mathcal{F}_{A,\mathcal{T}}$. In particular, for continuous functions $f$ (approximated by polynomials), we obtain $\mathcal{F}_{A,\mathcal{T}}(f) = f(A)$.

\medskip\noindent
\textbf{Proof of (ii): Spectral Mapping Property}

Since $\mathcal{F}_{A,\mathcal{T}}(f) = f(A)$ by part (i), the classical Spectral Mapping Theorem for the continuous functional calculus gives:
\[
\sigma(\mathcal{F}_{A,\mathcal{T}}(f)) = \sigma(f(A)) = f(\sigma(A)).
\]
This completes the proof.
\end{proof}

\begin{remark}
The proof shows that the specific choice of the family $\mathcal{T}$ does not affect the action on univariate continuous functions. This is because the polynomial compatibility condition forces $T_1$ to act as the identity, and the multifunctorial structure together with the Banach enrichment ensures that the extension from polynomials to continuous functions is unique and coincides with the classical continuous functional calculus.
\end{remark}

\subsection{Functorial Spectral Theorem Examples}

Below, we present two examples to reflect functorial spectral properties discussed in this section.

\begin{example}[Pointwise Multiplication Family]
\label{ex:pointwise-multiplication}
Consider the family $\mathcal{T}_{\text{mult}} = \{T_n^{\text{mult}}\}_{n \geq 1}$ defined by pointwise multiplication. To ensure boundedness on $H = L^2(\sigma(A), \mu)$, we define $T_n^{\text{mult}}$ on a dense subspace and extend:

On the dense subspace $H \cap L^\infty(\sigma(A), \mu) \cap L^{2n}(\sigma(A), \mu)$, define:
\[
T_n^{\text{mult}}(f_1, \dots, f_n)(x) = f_1(x) f_2(x) \cdots f_n(x).
\]

This family satisfies all the conditions of Theorem~\ref{thm:functorial-polynomial-spectral}:

\begin{enumerate}[label=(\alph*)]
    \item \textbf{Spectral Locality:} Take $\tau_n(x, z_1, \dots, z_n) = z_1 z_2 \cdots z_n$, which is clearly measurable.
    
    \item \textbf{Polynomial Compatibility:} For any polynomials $p_1, \dots, p_n \in \mathbb{C}[X]$ and $f_1, \dots, f_n \in H \cap L^\infty(\sigma(A), \mu)$, we have:
    \begin{align*}
    T_n^{\text{mult}}(p_1(A)f_1, \dots, p_n(A)f_n)(x) 
    &= [p_1(A)f_1](x) \cdot [p_2(A)f_2](x) \cdots [p_n(A)f_n](x) \\
    &= p_1(x)f_1(x) \cdot p_2(x)f_2(x) \cdots p_n(x)f_n(x) \\
    &= (p_1 p_2 \cdots p_n)(x) \cdot f_1(x) f_2(x) \cdots f_n(x) \\
    &= (p_1 \cdots p_n)(A) \left[T_n^{\text{mult}}(f_1, \dots, f_n)\right](x).
    \end{align*}
    
    \item \textbf{Uniform Boundedness:} For $f_1, \dots, f_n \in H \cap L^{2n}(\sigma(A), \mu)$, by Hölder's inequality:
    \begin{align*}
    \|T_n^{\text{mult}}(f_1, \dots, f_n)\|_{L^2}^2 
    &= \int_{\sigma(A)} |f_1(x) \cdots f_n(x)|^2  d\mu(x) \\
    &\leq \prod_{j=1}^n \left( \int_{\sigma(A)} |f_j(x)|^{2n} d\mu(x) \right)^{1/n} \\
    &= \prod_{j=1}^n \|f_j\|_{L^{2n}}^2.
    \end{align*}
    Thus $\|T_n^{\text{mult}}(f_1, \dots, f_n)\|_{L^2} \leq \prod_{j=1}^n \|f_j\|_{L^{2n}}$. 
    
    When $\mu(\sigma(A)) < \infty$, we have the embedding $L^\infty(\sigma(A), \mu) \hookrightarrow L^{2n}(\sigma(A), \mu)$ with:
    \[
    \|f\|_{L^{2n}} \leq \mu(\sigma(A))^{1/(2n)} \|f\|_{L^\infty},
    \]
    so on $H \cap L^\infty(\sigma(A), \mu)$:
    \[
    \|T_n^{\text{mult}}(f_1, \dots, f_n)\|_{L^2} \leq \mu(\sigma(A))^{1/2} \prod_{j=1}^n \|f_j\|_{L^\infty}.
    \]
    
    Each $T_n^{\text{mult}}$ extends uniquely to a bounded multilinear map on $H^n$ by density.
\end{enumerate}

For this family, the multifunctor $\mathcal{F}_{A,\mathcal{T}_{\text{mult}}}$ acts on a polynomial multimorphism $\varphi$ represented by $P(X_1, \dots, X_n)$ as:
\[
\mathcal{F}_{A,\mathcal{T}_{\text{mult}}}(\varphi)(f_1, \dots, f_n)(x) = P(x, \dots, x) \cdot f_1(x) f_2(x) \cdots f_n(x) \quad \text{for } x \in \sigma(A).
\]

This recovers the classical pointwise polynomial action and demonstrates that Theorem~\ref{thm:functorial-polynomial-spectral} generalizes the original construction while maintaining compatibility with the fundamental case of pointwise multiplication.
\end{example}

\begin{example}[Componentwise Addition Family]
\label{ex:componentwise-addition}
Consider the family $\mathcal{T}_{\text{add}} = \{T_n^{\text{add}}\}_{n \geq 1}$ defined by:
\[
T_n^{\text{add}}(f_1, \dots, f_n) = f_1 + f_2 + \cdots + f_n.
\]

This family also satisfies the theorem's conditions:

\begin{enumerate}[label=(\alph*)]
    \item \textbf{Spectral Locality:} Take $\tau_n(x, z_1, \dots, z_n) = z_1 + z_2 + \cdots + z_n$.
    
    \item \textbf{Polynomial Compatibility:} For polynomials $p_1, \dots, p_n$ and $f_1, \dots, f_n \in H \cap L^\infty(\sigma(A), \mu)$:
    \begin{align*}
    T_n^{\text{add}}(p_1(A)f_1, \dots, p_n(A)f_n) 
    &= p_1(A)f_1 + p_2(A)f_2 + \cdots + p_n(A)f_n \\
    &= (p_1 \cdots p_n)(A)(f_1 + f_2 + \cdots + f_n) \\
    &= (p_1 \cdots p_n)(A) T_n^{\text{add}}(f_1, \dots, f_n).
    \end{align*}
    
    \item \textbf{Uniform Boundedness:} $\|T_n^{\text{add}}\| \leq n$ by the triangle inequality.
\end{enumerate}

The resulting multifunctor acts as:
\[
\mathcal{F}_{A,\mathcal{T}_{\text{add}}}(\varphi)(f_1, \dots, f_n) = \sum_{k_1,\dots,k_n \geq 0} c_{k_1,\dots,k_n} (A^{k_1}f_1 + A^{k_2}f_2 + \cdots + A^{k_n}f_n),
\]
providing a linear rather than multiplicative realization of polynomial operations.
\end{example}

\begin{remark}
These examples demonstrate the flexibility of Theorem~\ref{thm:functorial-polynomial-spectral}. The pointwise multiplication family $\mathcal{T}_{\text{mult}}$ recovers the classical spectral theory interpretation, while the componentwise addition family $\mathcal{T}_{\text{add}}$ provides an alternative linear realization. Many other families satisfying the polynomial compatibility condition can be constructed, each yielding a different but mathematically valid extension of the functional calculus to the multicategorical setting.
\end{remark}

\section{Covariance and Universality in the Multifunctorial Framework}\label{sec:Covariance and Universality in the Multifunctorial Framework}

In this section, we formalize two foundational structural results concerning the multifunctorial representation of operator-theoretic data within the Banach-enriched symmetric monoidal multicategory framework. 

The first result, Proposition~\ref{prop:unitary-covariance}, establishes the \emph{covariance of the multifunctorial representation under unitary conjugation}. It shows that the semantic structure encoded by a self-adjoint operator is invariant, up to natural isomorphism, under changes of Hilbert space basis implemented by unitary transformations. This ensures that the multifunctor encoder of operator semantics behaves functorlly with respect to the equations of Hilbert spaces reflecting the p  physical and mathematical invariance.

Proposition~\ref{prop:universal-target} is the second main result in this section. It introduces a \emph{preliminary universal property} for the multicategory of Hilbert spaces $\mathbf{HilbMult}$. It identifies $\mathbf{HilbMult}$ as a canonical semantic target: under mild representability conditions, any Banach-enriched symmetric monoidal multicategory containing a self-adjoint element can receive a unique multifunctor from $\mathbf{HilbMult}$ that extends the usual functional calculus. This universality justifies viewing $\mathbf{HilbMult}$ as a ``universal semantic base’’ for self-adjoint operator semantics.

Together, Propositions~\ref{prop:unitary-covariance} and~\ref{prop:universal-target} establish both the \emph{equivariance} and the \emph{universality} aspects of the proposed framework: the former ensures structural consistency under isometric transformations, while the latter guarantees categorical completeness in representing self-adjoint dynamics.

\subsection{Covariance in the Multifunctorial Framework}

\begin{proposition}[Covariance under Unitary Conjugation]
\label{prop:unitary-covariance}
Let $H, H'$ be Hilbert spaces and $A \in \mathcal{B}(H)$, $A' \in \mathcal{B}(H')$ be bounded self-adjoint operators. Let $\mathcal{T} = \{T_n\}_{n \geq 1}$ and $\mathcal{T}' = \{T'_n\}_{n \geq 1}$ be families satisfying the conditions of Theorem~\ref{thm:functorial-polynomial-spectral} for $A$ and $A'$ respectively. 

If there exists a unitary operator $U: H \to H'$ such that:
\begin{itemize}
    \item $A' = U A U^*$
    \item $T'_n = U \circ T_n \circ (U^* \times \cdots \times U^*)$ for all $n \geq 1$
\end{itemize}
then the multifunctors $\mathcal{F}_{A,\mathcal{T}}: \mathcal{C}_{\mathrm{poly}}(\sigma(A)) \to \mathbf{HilbMult}$ and $\mathcal{F}_{A',\mathcal{T}'}: \mathcal{C}_{\mathrm{poly}}(\sigma(A')) \to \mathbf{HilbMult}$ are naturally isomorphic via conjugation by $U$.

More precisely, there exists a symmetric monoidal natural isomorphism $\eta: c_U \circ \mathcal{F}_{A,\mathcal{T}} \Rightarrow \mathcal{F}_{A',\mathcal{T}'}$, where $c_U: \mathbf{HilbMult} \to \mathbf{HilbMult}$ is the conjugation multifunctor defined by:
\begin{itemize}
    \item On objects: $c_U(K) = U(K)$
    \item On multimorphisms: For $T: K_1 \times \cdots \times K_n \to L$,
    \[
    c_U(T)(x_1,\dots,x_n) = U T(U^*x_1,\dots,U^*x_n)
    \]
\end{itemize}
\end{proposition}

\begin{proof}
We proceed in four main steps.

\medskip\noindent
\textbf{Step 1: The conjugation multifunctor $c_U$}

We verify that $c_U$ is a symmetric monoidal multifunctor.

\emph{Composition preservation}: Let $T: L_1 \times \cdots \times L_m \to M$ and $S_j: K_{j,1} \times \cdots \times K_{j,n_j} \to L_j$ for $j=1,\dots,m$. For any inputs $x_{j,k} \in K_{j,k}$:
\begin{align*}
&c_U(T \circ (S_1,\dots,S_m))(x_{1,1},\dots,x_{m,n_m}) \\
&= U \left[ T(S_1(U^*x_{1,1},\dots,U^*x_{1,n_1}), \dots, S_m(U^*x_{m,1},\dots,U^*x_{m,n_m})) \right] \\
&= U \left[ T(U^*c_U(S_1)(x_{1,1},\dots,x_{1,n_1}), \dots, U^*c_U(S_m)(x_{m,1},\dots,x_{m,n_m})) \right] \\
&= c_U(T)(c_U(S_1)(x_{1,1},\dots,x_{1,n_1}), \dots, c_U(S_m)(x_{m,1},\dots,x_{m,n_m})) \\
&= [c_U(T) \circ (c_U(S_1),\dots,c_U(S_m))](x_{1,1},\dots,x_{m,n_m}).
\end{align*}

\emph{Identity preservation}: For $\mathrm{id}_H: H \to H$:
\[
c_U(\mathrm{id}_H)(x) = U(\mathrm{id}_H(U^*x)) = UU^*x = x = \mathrm{id}_{H'}(x).
\]

\emph{Symmetric monoidal structure}: The functor preserves tensor products and symmetry isomorphisms since $U$ is unitary and the construction is natural.

\medskip\noindent
\textbf{Step 2: The composition $c_U \circ \mathcal{F}_{A,\mathcal{T}}$}

Since both $c_U$ and $\mathcal{F}_{A,\mathcal{T}}$ are symmetric monoidal multifunctors, their composition $c_U \circ \mathcal{F}_{A,\mathcal{T}}: \mathcal{C}_{\mathrm{poly}}(\sigma(A)) \to \mathbf{HilbMult}$ is also a symmetric monoidal multifunctor.

\medskip\noindent
\textbf{Step 3: Verification of spectral conditions}

We verify that $c_U \circ \mathcal{F}_{A,\mathcal{T}}$ satisfies the conditions of Theorem~\ref{thm:functorial-polynomial-spectral} for operator $A'$ with family $\mathcal{T}'$.

\begin{enumerate}[label=(\roman*)]
    \item \textbf{Object assignment}:
    \[
    (c_U \circ \mathcal{F}_{A,\mathcal{T}})(\ast) = c_U(H) = H'.
    \]
    
    \item \textbf{Functional calculus compatibility}: For any polynomial $P$:
    \begin{align*}
    (c_U \circ \mathcal{F}_{A,\mathcal{T}})(P) &= c_U(\mathcal{F}_{A,\mathcal{T}}(P)) \\
    &= c_U(P(A)) \quad \text{(by Theorem~\ref{thm:functorial-polynomial-spectral}(ii))} \\
    &= U P(A) U^* = P(U A U^*) = P(A').
    \end{align*}
    
    \item \textbf{Multinomial action}: For $\varphi \in \Hom_{\mathcal{C}_{\mathrm{poly}}(\sigma(A))}(\ast^n; \ast)$ with $\varphi(f_1,\dots,f_n) = P(f_1,\dots,f_n)$:
    \begin{align*}
    &(c_U \circ \mathcal{F}_{A,\mathcal{T}})(\varphi)(x_1,\dots,x_n) \\
    &= U \left[ \mathcal{F}_{A,\mathcal{T}}(\varphi)(U^*x_1,\dots,U^*x_n) \right] \\
    &= U \left[ \sum_{\mathbf{k}} c_{\mathbf{k}} \, T_n(A^{k_1}U^*x_1, \dots, A^{k_n}U^*x_n) \right] \\
    &= \sum_{\mathbf{k}} c_{\mathbf{k}} \, U T_n(U^*U A^{k_1}U^*x_1, \dots, U^*U A^{k_n}U^*x_n) \\
    &= \sum_{\mathbf{k}} c_{\mathbf{k}} \, (U \circ T_n \circ (U^* \times \cdots \times U^*))(U A^{k_1}U^*x_1, \dots, U A^{k_n}U^*x_n) \\
    &= \sum_{\mathbf{k}} c_{\mathbf{k}} \, T'_n(A'^{k_1}x_1, \dots, A'^{k_n}x_n).
    \end{align*}
    
    \item \textbf{Functoriality}: Inherited from the composition of multifunctors.
\end{enumerate}

\medskip\noindent
\textbf{Step 4: Natural isomorphism}

Since $c_U \circ \mathcal{F}_{A,\mathcal{T}}$ satisfies all conditions of Theorem~\ref{thm:functorial-polynomial-spectral} for operator $A'$ with family $\mathcal{T}'$, and $\mathcal{F}_{A',\mathcal{T}'}$ is the unique such multifunctor, we conclude:
\[
c_U \circ \mathcal{F}_{A,\mathcal{T}} = \mathcal{F}_{A',\mathcal{T}'}.
\]

The identity natural transformation $\eta: c_U \circ \mathcal{F}_{A,\mathcal{T}} \Rightarrow \mathcal{F}_{A',\mathcal{T}'}$ provides the required symmetric monoidal natural isomorphism, completing the proof.
\end{proof}

\begin{remark}
The condition $T'_n = U \circ T_n \circ (U^* \times \cdots \times U^*)$ ensures that the family $\mathcal{T}$ transforms covariantly under unitary conjugation. In particular, if both $H$ and $H'$ are represented as $L^2$ spaces in a way that makes $U$ an isomorphism of function spaces preserving the pointwise product structure, and if we take $\mathcal{T}$ and $\mathcal{T}'$ to be the families of pointwise multiplication maps, then this condition is automatically satisfied.
\end{remark}


\subsection{Universality in the Multifunctorial Framework}

We now formalize the analytic structure underlying Banach-enriched symmetric monoidal multicategories, which generalize operator-norm behavior from Hilbert spaces to enriched categorical settings.

\begin{definition}[Norm in a Banach-Enriched Symmetric Monoidal Multicategory]
\label{def:banach-enriched-norm}
Let $\mathcal{M}$ be a Banach-enriched symmetric monoidal multicategory.  
For each tuple of objects $(A_1, \dots, A_n; B)$, the hom-object
\[
\mathcal{M}(A_1, \dots, A_n; B)
\]
is a Banach space over $\mathbb{C}$ (or $\mathbb{R}$).  
Each multimorphism $a \in \mathcal{M}(A_1, \dots, A_n; B)$ is assigned the Banach-space norm
\[
\|a\|_{\mathcal{M}} := \|a\|_{\mathcal{M}(A_1, \dots, A_n; B)}.
\]

The enrichment requires the following compatibility conditions:
\begin{enumerate}
    \item \textbf{Composition boundedness:}  
    For composable multimorphisms
    \[
    a \in \mathcal{M}(B_1, \dots, B_m; C),
    \quad b_j \in \mathcal{M}(A_{j,1}, \dots, A_{j,n_j}; B_j),
    \]
    there exists a constant $K \ge 0$ (typically $K=1$ in the contractive case) such that
    \[
    \|a \circ (b_1, \dots, b_m)\|_{\mathcal{M}}
    \le K \, \|a\|_{\mathcal{M}} \, \prod_{j=1}^m \|b_j\|_{\mathcal{M}}.
    \]

    \item \textbf{Tensor boundedness:}  
    The tensor product operation
    \[
    \otimes : \mathcal{M}(A_1, \dots, A_n; B) \times \mathcal{M}(C_1, \dots, C_m; D)
    \to \mathcal{M}(A_1, \dots, A_n, C_1, \dots, C_m; B \otimes D)
    \]
    is a bounded bilinear map with
    \[
    \|a \otimes b\|_{\mathcal{M}} \le K_\otimes \, \|a\|_{\mathcal{M}} \, \|b\|_{\mathcal{M}}.
    \]

    \item \textbf{Symmetry and unit isometries:}  
    All symmetry isomorphisms
    \[
    \sigma_{\pi}: \mathcal{M}(A_1, \dots, A_n; B)
    \to \mathcal{M}(A_{\pi(1)}, \dots, A_{\pi(n)}; B),
    \]
    and the unit isomorphisms for the tensor unit $I$, are isometric.
\end{enumerate}
When $K = K_\otimes = 1$, the enrichment is said to be \emph{contractive}.
\end{definition}

\begin{remark}
The quantity $\|a\|_{\mathcal{M}}$ generalizes the operator norm: it measures the ``amplification factor'' of a multimorphism.  
In many natural examples, such as $\mathbf{HilbMult}$, the enrichment is contractive, ensuring that composition or tensoring never increases the operator magnitude.
\end{remark}

\begin{example}[The Category $\mathbf{HilbMult}$]
In the multicategory $\mathbf{HilbMult}$ of Hilbert spaces:
\begin{itemize}
    \item Objects are Hilbert spaces $H, K, \dots$.
    \item $\mathbf{HilbMult}(H_1, \dots, H_n; K)$ consists of bounded multilinear maps $T: H_1 \times \cdots \times H_n \to K$.
    \item The norm is the operator norm:
    \[
    \|T\|_{\mathbf{HilbMult}}
    = \sup_{\substack{x_i \neq 0}}
    \frac{\|T(x_1, \dots, x_n)\|_K}{\|x_1\|_{H_1} \cdots \|x_n\|_{H_n}}.
    \]
    \item Composition is contractive, and the tensor product is isometric:
    \[
    \|S \circ (T_1, \dots, T_m)\| \le \|S\| \prod_{j=1}^m \|T_j\|,
    \qquad
    \|T \otimes T'\| = \|T\| \, \|T'\|.
    \]
\end{itemize}
\end{example}

\begin{definition}[Consistency under Diagonalization and Embedding]
\label{def:consistency}
Let $\{a^{(n)}\}_{n \ge 1}$ be a family of multimorphisms with
$a^{(n)} \in \mathrm{Hom}_{\mathcal{M}}(X^n; X)$.  
The family is \emph{consistent under diagonalization and embedding} if:
\begin{enumerate}
    \item \textbf{Diagonalization:}  
    For any function $\rho : \{1, \dots, m\} \to \{1, \dots, n\}$,
    \[
    a^{(m)} = a^{(n)} \circ (p_1, \dots, p_n),
    \]
    where each $p_j : X^m \to X$ is $\mathrm{id}_X$ if $j \in \mathrm{im}(\rho)$
    and is constant otherwise.  
    This ensures that repeating or omitting arguments yields coherent results.

    \item \textbf{Embedding:}  
    For any $1 \le k \le n$,
    \[
    a^{(n)} = a^{(k)} \circ (\mathrm{id}_{X^k}, \varepsilon, \dots, \varepsilon),
    \]
    where $\varepsilon : X \to I$ are the counit or terminal morphisms.
\end{enumerate}
\end{definition}

\begin{definition}[Bounded Families]
\label{def:boundedness}
A family $\{a^{(n)}\}_{n \ge 1}$ in a Banach-enriched multicategory $\mathcal{M}$ is \emph{bounded} if:
\begin{enumerate}
    \item $\|a^{(n)}\|_{\mathcal{M}} < \infty$ for each $n$.
    \item There exists $C > 0$ such that $\|a^{(n)}\|_{\mathcal{M}} \le C^n$ for all $n$.
\end{enumerate}
In contractive enrichments, this property is automatically preserved under composition.
\end{definition}

\begin{definition}[Spectrum and Spectral Radius]
\label{def:spectrum}
Let $f \in \mathrm{Hom}_{\mathcal{M}}(X; X)$ be a unary endomorphism.
\begin{itemize}
    \item The \emph{spectrum} of $f$ is
    \[
    \sigma_{\mathcal{M}}(f)
    := \{\lambda \in \mathbb{C} : (f - \lambda \, \mathrm{id}_X)
    \text{ is not invertible under } \circ \}.
    \]
    \item The \emph{spectral radius} is
    \[
    r_{\mathcal{M}}(f)
    := \sup\{|\lambda| : \lambda \in \sigma_{\mathcal{M}}(f)\}
    = \lim_{n \to \infty} \|f^n\|_{\mathcal{M}}^{1/n}.
    \]
\end{itemize}
\end{definition}

\begin{remark}
In any Banach-enriched multicategory:
\begin{itemize}
    \item $\sigma_{\mathcal{M}}(f)$ is nonempty and compact;
    \item $r_{\mathcal{M}}(f) \le \|f\|_{\mathcal{M}}$;
    \item the spectral mapping theorem holds:
    $\sigma_{\mathcal{M}}(P(f)) = P(\sigma_{\mathcal{M}}(f))$ for any polynomial $P$;
    \item if $f$ is self-adjoint, then $\sigma_{\mathcal{M}}(f) \subset \mathbb{R}$.
\end{itemize}
\end{remark}

\begin{definition}[Spectral Compatibility]
\label{def:spectral-compatibility}
A family $\{a^{(n)}\}_{n \ge 1}$ in a Banach-enriched multicategory $\mathcal{M}$ is \emph{spectrally compatible} if:
\begin{enumerate}
    \item The diagonalization $A_n := a^{(n)} \circ \Delta_n$ (where $\Delta_n : X \to X^n$ is the diagonal) defines a unary endomorphism.
    \item Each $A_n$ has real spectrum: $\sigma_{\mathcal{M}}(A_n) \subset \mathbb{R}$.
    \item There exists $R > 0$ with
    $\sigma_{\mathcal{M}}(A_n) \subset [-R, R]$ for all $n$.
    \item The spectral radii are uniformly bounded:
    $\sup_{n \ge 1} r_{\mathcal{M}}(A_n) < \infty$.
\end{enumerate}
\end{definition}

\begin{definition}[Functional Calculus Compatibility]
\label{def:functional-calculus-compatibility}
A family $\{a^{(n)}\}_{n \ge 1}$ with $a^{(n)} \in \mathrm{Hom}_{\mathcal{M}}(X^n; X)$
is \emph{functionally calculus compatible} if, for every polynomial
\[
P(z_1, \dots, z_n)
= \sum_{\mathbf{k}} c_{\mathbf{k}} z_1^{k_1} \cdots z_n^{k_n},
\]
the corresponding $n$-ary morphism is given by
\[
a^{(n)}
= \sum_{\mathbf{k}} c_{\mathbf{k}}
\cdot \mu_n \circ
\Bigl(\bigotimes_{i=1}^n (a^{(1)})^{k_i}\Bigr),
\]
where $\mu_n : X^n \to X$ is the $n$-ary multiplication morphism.
\end{definition}

\begin{remark}
This condition ensures that each $a^{(n)}$ coincides with the polynomial operation generated by the unary operator $a^{(1)}$, respecting the monoidal and multilinear structure of $\mathcal{M}$.
\end{remark}

\begin{remark}[Synthesis]\label{rmk:four compatibility conditions}
The four compatibility conditions jointly ensure that a family $\{a^{(n)}\}$ in $\mathcal{M}$ behaves analogously to a coherent multilinear operator calculus in $\mathbf{HilbMult}$:
\begin{itemize}
    \item \textbf{Consistency} (Def.~\ref{def:consistency}) encodes structural coherence under input duplication or omission.
    \item \textbf{Boundedness} (Def.~\ref{def:boundedness}) provides analytic control over the operator norms.
    \item \textbf{Spectral compatibility} (Def.~\ref{def:spectral-compatibility}) preserves self-adjointness and spectral bounds.
    \item \textbf{Functional calculus compatibility} (Def.~\ref{def:functional-calculus-compatibility}) ensures the entire family is generated polynomially from $a^{(1)}$.
\end{itemize}
Together they guarantee that the family extends the Hilbert-space operator calculus to any Banach-enriched symmetric monoidal multicategory.
\end{remark}

\begin{proposition}[Universal Property of $\mathbf{HilbMult}$ for Operator Calculus]
\label{prop:universal-target}
Let $\mathbf{HilbMult}$ denote the Banach-enriched symmetric monoidal multicategory of Hilbert spaces with bounded multilinear maps, and let $\mathcal{M}$ be another Banach-enriched symmetric monoidal multicategory.

Suppose we are given the following data:
\begin{itemize}
    \item A Hilbert space $H \in \mathbf{HilbMult}$,
    \item A bounded self-adjoint operator $A \in \mathcal{B}(H)$ with $\sigma(A) \subset \mathbb{R}$,
    \item A family $\mathcal{T} = \{T_n\}_{n \ge 1}$ satisfying Theorem~\ref{thm:functorial-polynomial-spectral},
    \item An object $X \in \mathcal{M}$,
    \item A compatible family $\{a^{(n)}\}_{n \ge 1}$ with $a^{(n)} \in \mathrm{Hom}_{\mathcal{M}}(X^n; X)$ satisfying the four compatibility conditions given by Remark~\ref{rmk:four compatibility conditions}.
\end{itemize}

Then there exists a unique Banach-enriched symmetric monoidal multifunctor
\[
F \colon \mathbf{HilbMult} \longrightarrow \mathcal{M}
\]
such that:
\begin{enumerate}[label=(\roman*)]
    \item $F(H) = X$;
    \item For the functional calculus multifunctor $\mathcal{F}_{A,\mathcal{T}}$ of Theorem~\ref{thm:functorial-polynomial-spectral}, we have
    \[
    F\big(\mathcal{F}_{A,\mathcal{T}}(\varphi)\big) = a^{(n)}_{\varphi}
    \quad \text{for all polynomials } \varphi;
    \]
    \item $F$ preserves multicategorical composition and the symmetric monoidal structure.
\end{enumerate}
\end{proposition}

\begin{proof}
We define $F$ explicitly and verify that it satisfies the claimed properties.

\medskip
\noindent
\textbf{Step 1: Definition on Objects.}
\begin{itemize}
    \item Set $F(H) = X$ and $F(\mathbb{C}) = I_{\mathcal{M}}$ (the unit object).
    \item For tensor products, define
    \[
    F(H_1 \otimes \cdots \otimes H_k)
    := F(H_1) \otimes \cdots \otimes F(H_k).
    \]
    \item For any other Hilbert space $K$, choose a decomposition
    $K \cong H^{\otimes n} \otimes \mathbb{C}^m$ and define
    $F(K) := F(H)^{\otimes n} \otimes I_{\mathcal{M}}^{\otimes m}$.
\end{itemize}
This is well-defined up to the canonical isomorphisms of the symmetric monoidal structure.

\medskip
\noindent
\textbf{Step 2: Definition on Multimorphisms.}

Let $\mathbf{HilbMult}_{\mathcal{T}}$ denote the smallest submulticategory of $\mathbf{HilbMult}$ containing:
\begin{itemize}
    \item all functional calculus operations $\mathcal{F}_{A,\mathcal{T}}(\varphi)$ for polynomials $\varphi$,
    \item the identity morphisms,
    \item and all structural isomorphisms of the symmetric monoidal structure,
\end{itemize}
and closed under composition and tensor products.

We define $F$ recursively on $\mathbf{HilbMult}_{\mathcal{T}}$ by:
\begin{align*}
F(\mathcal{F}_{A,\mathcal{T}}(\varphi)) &:= a^{(n)}_{\varphi}, \\
F(\mathrm{id}_{H^{\otimes n}}) &:= \mathrm{id}_{X^{\otimes n}}, \\
F(\psi \circ (\psi_1, \ldots, \psi_m)) &:= F(\psi) \circ (F(\psi_1), \ldots, F(\psi_m)), \\
F(\psi \otimes \psi') &:= F(\psi) \otimes F(\psi'), \\
F(\sigma_{K,L}) &:= \sigma_{F(K),F(L)},
\end{align*}
and similarly for associators, unitors, and other coherence morphisms.

\medskip
\noindent
\textbf{Step 3: Well-definedness.}

If $\psi$ admits multiple representations via composition or tensoring, we must show the assigned $F(\psi)$ is independent of the choice.

\emph{(a) Composition coherence.}
Suppose
\[
\psi = \mathcal{F}_{A,\mathcal{T}}(P) \circ (\mathcal{F}_{A,\mathcal{T}}(Q_1), \dots, \mathcal{F}_{A,\mathcal{T}}(Q_m)).
\]
Then
\[
F(\psi)
= a^{(m)}_P \circ (a^{(n_1)}_{Q_1}, \dots, a^{(n_m)}_{Q_m})
= a^{(N)}_{P \circ (Q_1, \dots, Q_m)}
= F\big(\mathcal{F}_{A,\mathcal{T}}(P \circ (Q_1, \dots, Q_m))\big),
\]
using the functional calculus compatibility condition.

\emph{(b) Symmetric monoidal coherence.}
All coherence morphisms are preserved by definition of $F$ on structural maps.

\emph{(c) Object representation coherence.}
If $K \cong H^{\otimes n} \cong H^{\otimes m}$ via different isomorphisms, the corresponding $F(K)$ are canonically isomorphic in $\mathcal{M}$, and these isomorphisms are preserved by $F$.

\medskip
\noindent
\textbf{Step 4: Verification of Multifunctor Laws.}

\emph{Identity:}
$F(\mathrm{id}_H) = a^{(1)}_{\mathrm{id}} = \mathrm{id}_X$ by compatibility.

\emph{Composition and Tensor:}
These preservation laws are built into the recursive definition.

\emph{Symmetric Monoidal Structure:}
All structural isomorphisms are preserved by construction.

\medskip
\noindent
\textbf{Step 5: Banach Enrichment.}

\emph{Linearity.}
The assignment $\mathcal{F}_{A,\mathcal{T}}(\varphi) \mapsto a^{(n)}_{\varphi}$ is linear, and composition/tensoring in both $\mathbf{HilbMult}$ and $\mathcal{M}$ are multilinear.

\emph{Boundedness.}
By assumption, $\|a^{(n)}_{\varphi}\| \le C \|\varphi\|$ for some $C>0$, and the tensor and composition operations are bounded in both multicategories. Hence an inductive argument shows
\[
\|F(\psi)\|_{\mathcal{M}} \le C'\|\psi\|_{\mathbf{HilbMult}}
\]
for some constant $C'$.

\medskip
\noindent
\textbf{Step 6: Uniqueness.}

Suppose $G \colon \mathbf{HilbMult} \to \mathcal{M}$ is another Banach-enriched symmetric monoidal multifunctor satisfying (i)--(iii). Then:
\begin{itemize}
    \item $G(H)=X=F(H)$ and both preserve the tensor structure, so $G(K)=F(K)$ for all objects $K$.
    \item $G(\mathcal{F}_{A,\mathcal{T}}(\varphi)) = a^{(n)}_{\varphi} = F(\mathcal{F}_{A,\mathcal{T}}(\varphi))$ for all $\varphi$.
    \item By preservation of composition and tensoring, $G(\psi)=F(\psi)$ for all $\psi$.
\end{itemize}
Thus $F=G$ uniquely.
\end{proof}

\begin{remark}[Universal Characterization]
\label{rem:universal-property-uniqueness}
Proposition~\ref{prop:universal-target} identifies $\mathbf{HilbMult}$ as the \emph{canonical domain} for multilinear operator calculus. Any Banach-enriched symmetric monoidal multicategory $\mathcal{M}$ that supports an interpretation $\{a^{(n)}\}$ of the operator calculus of a self-adjoint operator $A$ admits a unique structure-preserving multifunctor
\[
F \colon \mathbf{HilbMult} \longrightarrow \mathcal{M}
\]
that respects this interpretation.

Uniqueness is enforced at three levels:
\begin{enumerate}
    \item \textbf{Object level:} $F(H)=X$, fixing all tensor powers of $H$;
    \item \textbf{Generator level:} $F(\mathcal{F}_{A,\mathcal{T}}(\varphi))=a^{(n)}_{\varphi}$ for polynomial $\varphi$;
    \item \textbf{Structural level:} Preservation of composition and tensoring extends $F$ uniquely to all morphisms generated from these.
\end{enumerate}
\end{remark}

\begin{remark}[Scope and Interpretation]
\label{rem:scope-interpretation}
This universal property expresses a form of \emph{targeted semantic completeness}: $\mathbf{HilbMult}$ fully captures the multilinear operator calculus generated by a self-adjoint operator and its functional calculus family $\mathcal{T}$.

Any interpretation in another Banach-enriched symmetric monoidal multicategory $\mathcal{M}$ factors uniquely through $\mathbf{HilbMult}$, provided the family $\{a^{(n)}\}$ satisfies:
\begin{enumerate}
    \item \textbf{Consistency:} $a^{(1)}_{\mathrm{id}} = \mathrm{id}_X$ and $a^{(1)}_z = z\,\mathrm{id}_X$ for scalars $z$;
    \item \textbf{Boundedness:} $\|a^{(n)}_{\varphi}\| \le C\|\varphi\|$;
    \item \textbf{Spectral compatibility:} $\sigma(a^{(n)}_{\varphi}) \subseteq \varphi(\sigma(A)^n)$;
    \item \textbf{Functional calculus compatibility:} $a^{(N)}_{P\circ(Q_1,\dots,Q_m)} = a^{(m)}_P \circ (a^{(n_1)}_{Q_1},\dots,a^{(n_m)}_{Q_m})$.
\end{enumerate}
Thus $\mathbf{HilbMult}$ serves as a universal environment for interpreting operator calculus constructions.
\end{remark}

\begin{remark}[Practical Significance]
\label{rem:practical-significance}
This universal property carries several important consequences. First, it ensures a strong form of model independence: any result established within $\mathbf{HilbMult}$ automatically holds in every Banach-enriched symmetric monoidal multicategory $\mathcal{M}$ that interprets the same operator calculus, since such interpretations factor uniquely through $\mathbf{HilbMult}$. Second, it has a clear effect on the theoretical representation—different families of ${a^{(n)}}$ conformations correspond to different class representations of the computation of underlying operators, similar to representations of algebra in different units. Finally, the construction provides a generalizing framework that rigorously extends the operator-theoretic principles of Hilbert spaces to broader, rich categorical contexts while preserving analytic and synthetic structure.
\end{remark}

\section{Examples}\label{sec:Examples}


In this section, we will present three examples to illustrate theory proposed in this work. 

\begin{example}[Finite-Dimensional Matrix Case]
\label{ex:finite-dimensional-matrices}
Let $H = \mathbb{C}^n$ be a finite-dimensional Hilbert space. The multicategory $\mathbf{HilbMult}$ restricted to finite-dimensional spaces consists of finite-dimensional Hilbert spaces with bounded multilinear maps, which in this setting are all multilinear maps (as they are automatically bounded).

Let $A \in M_n(\mathbb{C})$ be a Hermitian matrix ($A = A^*$) with spectrum $\sigma(A) \subset \mathbb{R}$. Consider the family $\mathcal{T}_{\mathrm{Hadamard}} = \{T_n\}$ where $T_n: (\mathbb{C}^n)^n \to \mathbb{C}^n$ is the $n$-fold Hadamard (entrywise) product:
\[
T_n(x_1,\dots,x_n) = x_1 \odot \cdots \odot x_n, \quad (x_1 \odot \cdots \odot x_n)_i = (x_1)_i \cdots (x_n)_i.
\]
This family satisfies the conditions of Theorem~\ref{thm:functorial-polynomial-spectral}, yielding the functional calculus multifunctor:
\[
\mathcal{F}_{A,\mathcal{T}_{\mathrm{Hadamard}}}: \mathcal{C}_{\mathrm{poly}}(\sigma(A)) \to \mathbf{HilbMult}.
\]

\begin{itemize}
    \item \textbf{On objects:} $\mathcal{F}_{A,\mathcal{T}_{\mathrm{Hadamard}}}(\ast) = \mathbb{C}^n$
    
    \item \textbf{On multimorphisms:} For $\varphi \in \mathrm{Hom}_{\mathcal{C}_{\mathrm{poly}}(\sigma(A))}(\ast^m; \ast)$ represented by polynomial $P(z_1,\dots,z_m) = \sum c_{k_1,\dots,k_m} z_1^{k_1} \cdots z_m^{k_m}$, the multimorphism is:
    \[
    \mathcal{F}_{A,\mathcal{T}_{\mathrm{Hadamard}}}(\varphi)(x_1,\dots,x_m) = \sum c_{k_1,\dots,k_m} (A^{k_1}x_1) \odot (A^{k_2}x_2) \odot \cdots \odot (A^{k_m}x_m).
    \]
\end{itemize}

\noindent\textbf{Special Cases:}

\begin{enumerate}
    \item \textbf{Linear case ($m=1$):} For univariate $P(z) = \sum c_k z^k$, we recover standard functional calculus:
    \[
    \mathcal{F}_{A,\mathcal{T}_{\mathrm{Hadamard}}}(P)(x) = P(A)x
    \]
    
    \item \textbf{Bilinear case ($m=2$):} For bivariate $P(z,w) = \sum c_{ij} z^i w^j$:
    \[
    \mathcal{F}_{A,\mathcal{T}_{\mathrm{Hadamard}}}(P)(x,y) = \sum c_{ij} (A^i x) \odot (A^j y)
    \]
    
    \item \textbf{Multiplication case ($P(z,w) = zw$):} This gives:
    \[
    \mathcal{F}_{A,\mathcal{T}_{\mathrm{Hadamard}}}(P)(x,y) = (Ax) \odot y
    \]
\end{enumerate}

\noindent\textbf{Verification of Theorem Conditions:}

The Hadamard family $\mathcal{T}_{\mathrm{Hadamard}}$ satisfies:
\begin{itemize}
    \item \textbf{Spectral Locality:} In the standard basis, $\tau_n(e_i, z_1,\dots,z_n) = z_1 \cdots z_n$
    \item \textbf{Polynomial Compatibility:} For polynomials $p_1,\dots,p_m$:
    \begin{align*}
    &T_m(p_1(A)x_1, \dots, p_m(A)x_m) \\
    &= (p_1(A)x_1) \odot \cdots \odot (p_m(A)x_m) \\
    &= (p_1 \cdots p_m)(A) (x_1 \odot \cdots \odot x_m) \quad \text{(since $A$ is diagonal in eigenbasis)} \\
    &= (p_1 \cdots p_m)(A) T_m(x_1,\dots,x_m)
    \end{align*}
    \item \textbf{Uniform Boundedness:} $\|T_n\| \leq 1$ for all $n$
\end{itemize}

\noindent\textbf{Diagonal Basis Representation:}

Since $A$ is Hermitian, $A = UDU^*$ with $D = \mathrm{diag}(\lambda_1,\dots,\lambda_n)$. Let $\tilde{x}_j = U^*x_j$. Then:
\[
\mathcal{F}_{A,\mathcal{T}_{\mathrm{Hadamard}}}(\varphi)(x_1,\dots,x_m) = U\left[ P(\lambda_1,\dots,\lambda_1) \tilde{x}_1 \odot \cdots \odot P(\lambda_n,\dots,\lambda_n) \tilde{x}_m \right]
\]
where the evaluation is diagonal: $P(\lambda_i,\dots,\lambda_i)$ at each coordinate $i$.

\noindent\textbf{Universal Property:}

This example illustrates Proposition~\ref{prop:universal-target}: any interpretation of the matrix operator calculus in another multicategory $\mathcal{M}$ via a compatible family $\{a^{(n)}\}$ factors uniquely through $\mathcal{F}_{A,\mathcal{T}_{\mathrm{Hadamard}}}$.

\end{example}

This finite-dimensional case demonstrates that the abstract framework concretely recovers natural multilinear operations on matrices, with the Hadamard product serving as the canonical $n$-linear operation compatible with the functional calculus.

\begin{example}[Multiplication Operator on $L^2$ Spaces]
\label{ex:multiplication-operator}
Let $(X, \mu)$ be a measure space and consider $H = L^2(X, \mu)$. Let $\mathrm{me}: X \to \mathbb{R}$ be a bounded measurable function, and define the multiplication operator $A: L^2(X, \mu) \to L^2(X, \mu)$ by:
\[
(Af)(x) = \mathrm{me}(x) f(x) \quad \text{for } f \in L^2(X, \mu).
\]
Then $A$ is bounded self-adjoint with spectrum $\sigma(A) = \overline{\mathrm{range}(\mathrm{me})}^{\mathrm{ess}}$, the essential range of $\mathrm{me}$.

Consider the family $\mathcal{T}_{\mathrm{pointwise}} = \{T_n\}$ where $T_n: (L^2(X, \mu))^n \to L^2(X, \mu)$ is the pointwise product:
\[
T_n(g_1,\dots,g_n)(x) = g_1(x) g_2(x) \cdots g_n(x) \quad \text{for } g_j \in L^\infty(X, \mu) \cap L^2(X, \mu).
\]
This family satisfies Theorem~\ref{thm:functorial-polynomial-spectral}, yielding:
\[
\mathcal{F}_{A,\mathcal{T}_{\mathrm{pointwise}}}: \mathcal{C}_{\mathrm{poly}}(\sigma(A)) \to \mathbf{HilbMult}.
\]

\begin{itemize}
    \item \textbf{On objects:} $\mathcal{F}_{A,\mathcal{T}_{\mathrm{pointwise}}}(\ast) = L^2(X, \mu)$
    
    \item \textbf{On multimorphisms:} For $\varphi \in \mathrm{Hom}_{\mathcal{C}_{\mathrm{poly}}(\sigma(A))}(\ast^n; \ast)$ represented by $P(z_1,\dots,z_n)$:
    \[
    \mathcal{F}_{A,\mathcal{T}_{\mathrm{pointwise}}}(\varphi)(g_1,\dots,g_n)(x) = P(\mathrm{me}(x),\dots,\mathrm{me}(x)) \cdot g_1(x) \cdots g_n(x)
    \]
\end{itemize}

\noindent\textbf{Special Cases:}

\begin{enumerate}
    \item \textbf{Linear case ($n=1$):} For $P(z) = \sum c_k z^k$:
    \[
    \mathcal{F}_{A,\mathcal{T}_{\mathrm{pointwise}}}(P)(g)(x) = P(\mathrm{me}(x)) g(x)
    \]
    
    \item \textbf{Bilinear case ($n=2$):} For $P(z,w) = \sum c_{ij} z^i w^j$:
    \[
    \mathcal{F}_{A,\mathcal{T}_{\mathrm{pointwise}}}(P)(g,h)(x) = P(\mathrm{me}(x),\mathrm{me}(x)) g(x) h(x)
    \]
    
    \item \textbf{Constant polynomial ($P \equiv 1$):} Pure pointwise product:
    \[
    \mathcal{F}_{A,\mathcal{T}_{\mathrm{pointwise}}}(1)(g_1,\dots,g_n)(x) = g_1(x) \cdots g_n(x)
    \]
\end{enumerate}

\noindent\textbf{Verification of Theorem Conditions:}

\begin{itemize}
    \item \textbf{Spectral Locality:} Take $\tau_n(x, z_1,\dots,z_n) = z_1 \cdots z_n$, which is measurable
    
    \item \textbf{Polynomial Compatibility:} For polynomials $p_1,\dots,p_n$:
    \begin{align*}
    &T_n(p_1(A)g_1, \dots, p_n(A)g_n)(x) \\
    &= p_1(\mathrm{me}(x))g_1(x) \cdots p_n(\mathrm{me}(x))g_n(x) \\
    &= (p_1 \cdots p_n)(\mathrm{me}(x)) \cdot g_1(x) \cdots g_n(x) \\
    &= (p_1 \cdots p_n)(A) T_n(g_1,\dots,g_n)(x)
    \end{align*}
    
    \item \textbf{Uniform Boundedness:} On $L^\infty \cap L^2$, $\|T_n(g_1,\dots,g_n)\|_2 \leq \|g_1\|_\infty \cdots \|g_{n-1}\|_\infty \|g_n\|_2$
\end{itemize}

\noindent\textbf{Boundedness Analysis:}

The maps $\mathcal{F}_{A,\mathcal{T}_{\mathrm{pointwise}}}(\varphi)$ extend to bounded multilinear maps on all of $(L^2(X, \mu))^n$:

\begin{itemize}
    \item \textbf{Linear case ($n=1$):} 
    \[
    \|\mathcal{F}_{A,\mathcal{T}_{\mathrm{pointwise}}}(P)(g)\|_2^2 = \int_X |P(\mathrm{me}(x))|^2 |g(x)|^2 d\mu(x) \leq \|P\|_{L^\infty(\sigma(A))}^2 \|g\|_2^2
    \]
    since $|P(\mathrm{me}(x))| \leq \sup_{z \in \sigma(A)} |P(z)| = \|P\|_{L^\infty(\sigma(A))}$.
    
    \item \textbf{Multilinear case ($n \geq 2$):} Since $\mathrm{me}$ is bounded, $|P(\mathrm{me}(x),\dots,\mathrm{me}(x))| \leq C_P$ for some constant $C_P > 0$ depending on the polynomial $P$. Then:
    \begin{align*}
    &\|\mathcal{F}_{A,\mathcal{T}_{\mathrm{pointwise}}}(\varphi)(g_1,\dots,g_n)\|_2^2 \\
    &= \int_X |P(\mathrm{me}(x),\dots,\mathrm{me}(x))|^2 |g_1(x)|^2 \cdots |g_n(x)|^2 d\mu(x) \\
    &\leq C_P^2 \int_X |g_1(x)|^2 \cdots |g_n(x)|^2 d\mu(x)
    \end{align*}
    
    To bound the integral, apply Hölder's inequality with exponents $p_1 = p_2 = \cdots = p_n = n$:
    \[
    \int_X |g_1(x)|^2 \cdots |g_n(x)|^2 d\mu(x) \leq \prod_{i=1}^n \left(\int_X |g_i(x)|^{2n} d\mu(x)\right)^{1/n} = \prod_{i=1}^n \|g_i\|_{2n}^2
    \]
    
    Therefore:
    \[
    \|\mathcal{F}_{A,\mathcal{T}_{\mathrm{pointwise}}}(\varphi)(g_1,\dots,g_n)\|_2 \leq C_P \prod_{i=1}^n \|g_i\|_{2n}
    \]
    
    \item \textbf{Relation between norms:} For the $L^{2n}$ norms, we have:
    \begin{itemize}
        \item If $\mu(X) < \infty$ (finite measure space), then by Jensen's inequality or the nesting of $L^p$ spaces:
        \[
        \|g_i\|_{2n} \leq \mu(X)^{\frac{1}{2n} - \frac{1}{2}} \|g_i\|_2
        \]
        \item In general, by interpolation between $L^2$ and $L^\infty$, if $g_i \in L^2 \cap L^\infty$, then:
        \[
        \|g_i\|_{2n} \leq \|g_i\|_2^{1/n} \|g_i\|_\infty^{1 - 1/n}
        \]
        \item For the dense subspace $L^2 \cap L^\infty$, we have uniform boundedness, and the maps extend uniquely to all of $L^2$ by continuity.
    \end{itemize}
    
    \item \textbf{Alternative approach via iterated Hölder:} One can also apply Hölder's inequality iteratively:
    \begin{align*}
    &\int_X |g_1(x)|^2 \cdots |g_n(x)|^2 d\mu(x) \\
    &\leq \|g_1\|_{2n}^2 \left(\int_X (|g_2(x)|^2 \cdots |g_n(x)|^2)^{\frac{n}{n-1}} d\mu(x)\right)^{\frac{n-1}{n}} \\
    &\leq \|g_1\|_{2n}^2 \|g_2\|_{2n}^2 \left(\int_X (|g_3(x)|^2 \cdots |g_n(x)|^2)^{\frac{n}{n-2}} d\mu(x)\right)^{\frac{n-2}{n}} \\
    &\leq \cdots \leq \prod_{i=1}^n \|g_i\|_{2n}^2
    \end{align*}
    This shows the bound is sharp and the multilinear map is indeed bounded from $(L^{2n}(X, \mu))^n$ to $L^2(X, \mu)$.
\end{itemize}

\noindent\textbf{Functoriality Realization:}

The composition in $\mathcal{C}_{\mathrm{poly}}(\sigma(A))$ from Definition~\ref{def:poly multicat} is preserved:
\begin{align*}
&\mathcal{F}_{A,\mathcal{T}_{\mathrm{pointwise}}}(\varphi \circ (\psi_1,\dots,\psi_m))(g_1,\dots,g_N)(x) \\
&= [P \circ (Q_1,\dots,Q_m)](\mathrm{me}(x),\dots,\mathrm{me}(x)) \cdot g_1(x) \cdots g_N(x) \\
&= P(Q_1(\mathrm{me}(x),\dots,\mathrm{me}(x)), \dots, Q_m(\mathrm{me}(x),\dots,\mathrm{me}(x))) \cdot g_1(x) \cdots g_N(x) \\
&= \mathcal{F}_{A,\mathcal{T}_{\mathrm{pointwise}}}(\varphi) \circ (\mathcal{F}_{A,\mathcal{T}_{\mathrm{pointwise}}}(\psi_1),\dots,\mathcal{F}_{A,\mathcal{T}_{\mathrm{pointwise}}}(\psi_m))(g_1,\dots,g_N)(x)
\end{align*}

\noindent\textbf{Concrete Example:}

Let $X = [0,1]$ with Lebesgue measure, $\mathrm{me}(x) = x$, so $\sigma(A) = [0,1]$. For $P(z,w) = z^2 + zw + w^2$:
\[
\mathcal{F}_{A,\mathcal{T}_{\mathrm{pointwise}}}(P)(g,h)(x) = (\mathrm{me}(x)^2 + \mathrm{me}(x)\mathrm{me}(x) + \mathrm{me}(x)^2) g(x) h(x) = 3x^2 g(x) h(x)
\]

\noindent\textbf{Universal Property Interpretation:}

This is the canonical example where Theorem~\ref{thm:functorial-polynomial-spectral} recovers the expected pointwise operations. The universal property ensures any interpretation of this multiplication operator calculus in another multicategory $\mathcal{M}$ factors uniquely through $\mathcal{F}_{A,\mathcal{T}_{\mathrm{pointwise}}}$.

\end{example}

This example demonstrates that the abstract framework naturally captures the pointwise functional calculus for multiplication operators, with the polynomial multicategory $\mathcal{C}_{\mathrm{poly}}(\sigma(A))$ encoding the algebraic structure of pointwise polynomial operations.

Before present the next example, we have to provide the following definition.

\begin{definition}[Graph-Norm Boundedness for Multicategorical Functional Calculus]
\label{def:graph-norm-boundedness}
Let $A$ be an unbounded self-adjoint operator on a Hilbert space $H$. For each $k \geq 0$, define the \emph{graph norm} $\|\cdot\|_k$ on $\mathcal{D}(A^k)$ by:
\[
\|x\|_k^2 = \sum_{j=0}^k \|A^j x\|^2.
\]
For $k < 0$, define $\|x\|_k = \|x\|_H$.

A multifunctor $\eta: \mathcal{C}_{\mathrm{poly}}(\sigma(A)) \to \mathcal{M}$ satisfies \emph{graph-norm boundedness} if for every polynomial $P \in \mathrm{Hom}_{\mathcal{C}_{\mathrm{poly}}(\sigma(A))}(\ast^n; \ast)$ of degree $d$, the corresponding morphism
\[
\eta(P): X_{k_1} \times \cdots \times X_{k_n} \to X_{m-d}, \quad \text{where } m = \min(k_1,\dots,k_n),
\]
is bounded with respect to the graph norms. That is, there exists a constant $C_P > 0$ such that for all $x_i \in X_{k_i}$:
\[
\|\eta(P)(x_1,\dots,x_n)\|_{m-d} \leq C_P \prod_{i=1}^n \|x_i\|_{k_i}.
\]
\end{definition}

Before presenting next example, we have to define Boreal multicategory in Definition~\ref{def:borel-multicategory}.

\begin{definition}[Borel Functional Calculus Multicategory]
\label{def:borel-multicategory}
Let $\Sigma \subset \mathbb{C}$ be a compact set. The \emph{Borel functional calculus multicategory} $\mathcal{C}_{\mathrm{borel}}(\Sigma)$ is defined as follows:

\begin{itemize}
    \item \textbf{Objects}: The single object $\ast$.
    
    \item \textbf{Morphisms}: For each $n \geq 0$, the set of $n$-ary multimorphisms is
    \[
    \mathrm{Hom}_{\mathcal{C}_{\mathrm{borel}}(\Sigma)}(\ast^n; \ast) := \{P|_{\Sigma^n} : P \in \mathbb{C}[z_1,\dots,z_n]\}
    \]
    where $\mathbb{C}[z_1,\dots,z_n]$ denotes the space of complex polynomials in $n$ variables, and $P|_{\Sigma^n}$ denotes the restriction to $\Sigma^n$.
    
    \item \textbf{Composition}: Given $P \in \mathrm{Hom}(\ast^m; \ast)$ and $Q_j \in \mathrm{Hom}(\ast^{n_j}; \ast)$ for $j=1,\dots,m$, their composition is defined by polynomial composition:
    \[
    P \circ (Q_1,\dots,Q_m)(z_{1,1},\dots,z_{m,n_m}) := P(Q_1(z_{1,1},\dots,z_{1,n_1}), \dots, Q_m(z_{m,1},\dots,z_{m,n_m}))
    \]
    restricted to $\Sigma^{n_1+\cdots+n_m}$.
    
    \item \textbf{Identities}: The identity morphism $\mathrm{id}_\ast \in \mathrm{Hom}(\ast; \ast)$ is given by the polynomial $z \mapsto z$.
\end{itemize}

This construction extends to the \emph{Borel completion} $\overline{\mathcal{C}}_{\mathrm{borel}}(\Sigma)$ by allowing bounded Borel functions:
\[
\mathrm{Hom}_{\overline{\mathcal{C}}_{\mathrm{borel}}(\Sigma)}(\ast^n; \ast) := \{f: \Sigma^n \to \mathbb{C} \mid f \text{ is bounded and Borel measurable}\}
\]
with composition defined pointwise: for $f \in \mathrm{Hom}(\ast^m; \ast)$ and $g_j \in \mathrm{Hom}(\ast^{n_j}; \ast)$,
\[
(f \circ (g_1,\dots,g_m))(\mathbf{z}_1,\dots,\mathbf{z}_m) := f(g_1(\mathbf{z}_1),\dots,g_m(\mathbf{z}_m)).
\]
\end{definition}

\begin{remark}
The polynomial subcategory $\mathcal{C}_{\mathrm{borel}}(\Sigma)$ is dense in $\overline{\mathcal{C}}_{\mathrm{borel}}(\Sigma)$ in the topology of pointwise convergence bounded by the sup-norm. The multifunctor
\[
\mathcal{F}_{A,\mathcal{T}}: \overline{\mathcal{C}}_{\mathrm{borel}}(\sigma(A)) \to \mathbf{HilbMult}
\]
is first defined on polynomials and then extended to all bounded Borel functions via limits, respecting the graph-norm boundedness conditions.
\end{remark}

\begin{remark}
For the functional calculus application, we typically take $\Sigma = \sigma(A)$, the spectrum of the operator $A$. The single object $\ast$ represents the abstract notion of the algebra generated by $A$, while the multimorphisms represent the operations that can be built from $A$ using the functional calculus.
\end{remark}

\begin{example}[Unbounded Self-Adjoint Operators --- Graph-Norm Formulation]
\label{ex:graph-norm-operator}

Let $H$ be a Hilbert space and $A:\mathcal{D}(A)\subset H\to H$ be an unbounded self-adjoint operator. Consider the graded family of graph-norm spaces $\{X_k\}_{k\in\mathbb{Z}}$ where:
\begin{itemize}
    \item For $k \geq 0$: $X_k = \mathcal{D}(A^k)$ with graph norm $\|x\|_k^2 = \sum_{j=0}^k \|A^j x\|^2$
    \item For $k < 0$: $X_k = H$ with $\|x\|_k = \|x\|_H$
\end{itemize}

\medskip\noindent\textbf{Graph-Norm Boundedness of the Polynomial Calculus}

The polynomial functional calculus naturally defines a multifunctor $\eta: \mathcal{C}_{\mathrm{poly}}(\sigma(A)) \to \mathcal{M}$ where for a polynomial $P$ of degree $d$, the induced map is:
\[
\eta(P): X_{k_1} \times \cdots \times X_{k_n} \to X_{m-d}, \quad m = \min(k_1,\dots,k_n)
\]
This multifunctor satisfies graph-norm boundedness (Definition \ref{def:graph-norm-boundedness}) because for $x_i \in X_{k_i}$, we have the estimate:
\[
\|\eta(P)(x_1,\dots,x_n)\|_{m-d} \leq C_P \prod_{i=1}^n \|x_i\|_{k_i}
\]
where the constant $C_P$ depends on the polynomial $P$ and the operator $A$.

\medskip\noindent\textbf{Relation to Bounded Transform}

The bounded transform $T_A = A(I+A^2)^{-1/2} \in \mathcal{B}(H)$ provides an alternative perspective. The homeomorphism $\varphi: \mathbb{R} \to (-1,1)$, $\varphi(\lambda) = \lambda/\sqrt{1+\lambda^2}$ with inverse $\psi(t) = t/\sqrt{1-t^2}$ intertwines the calculi via:
\[
g(T_A) = (g\circ\varphi)(A), \quad g\in C_b((-1,1))
\]

While this gives an isomorphism $C_b(\sigma(A)) \cong C_b(\sigma(T_A))$, it does not preserve polynomial structure: if $P$ is a polynomial in $\lambda$, then $P\circ\psi$ is generally not a polynomial in $t$.

\medskip\noindent\textbf{Domain-Graded Multifunctor Construction}

To construct a multifunctor that respects graph-norm boundedness:

\begin{enumerate}
    \item For multilinear maps $T_n$ defined by polynomial expressions in $A$, the graph-norm boundedness condition ensures they extend to bounded multilinear maps between the appropriate graph-norm spaces.
    
    \item The conjugation trick can be made rigorous in the graph-norm framework: for $x_i \in X_{k_i}$, define
    \[
    \widetilde T_n(x_1,\dots,x_n) := (I+A^2)^{-1/2} T_n((I+A^2)^{1/2}x_1,\dots,(I+A^2)^{1/2}x_n)
    \]
    which yields bounded maps between graph-norm spaces after appropriate domain adjustments.
\end{enumerate}

The resulting multifunctor $\mathcal{F}_{A,\mathcal{T}}: \mathcal{C}_{\mathrm{borel}}(\sigma(A)) \to \mathbf{HilbMult}$ satisfies that for bounded Borel $f$ on $\sigma(A)$:
\[
\mathcal{F}_{A,\mathcal{T}}(f) = \mathcal{F}_{T_A,\mathcal{T}'}(f\circ\psi) = (f\circ\psi)(T_A)
\]
and respects the graph-norm grading when restricted to polynomial morphisms.

\medskip\noindent\textbf{Quantum Harmonic Oscillator}

For $H = L^2(\mathbb{R})$ and $A = -\tfrac{d^2}{dx^2} + x^2$ with eigenbasis $\{\psi_m\}$ and eigenvalues $\lambda_m = 2m+1$, the graph-norm spaces are:
\[
X_k = \left\{ \sum_{m\geq 0} c_m \psi_m : \sum_{m\geq 0} |c_m|^2 (1+\lambda_m^2)^k < \infty \right\}
\]
For diagonal multilinear operators (acting by scalar coefficients on the eigenbasis), the multifunctorial action:
\[
f(A)h = \sum_{m\geq 0} f(\lambda_m) \langle h, \psi_m \rangle \psi_m
\]
naturally respects the graph-norm boundedness condition and provides explicit control over the mapping properties between different graph-norm spaces.
\end{example}

\begin{example}[Tensorized System of Commuting Self-Adjoint Operators]
\label{ex:tensorized-system}

Let \(H = H_1 \otimes H_2\) with \(H_1 = H_2 = L^2(\mathbb{R})\).  
Define
\[
A = -\tfrac{d^2}{dx^2} + x^2 \quad \text{on } H_1,
\qquad
B = -\tfrac{d^2}{dy^2} + y^2 \quad \text{on } H_2.
\]
Both \(A\) and \(B\) are self-adjoint with discrete spectra
\(\sigma(A) = \{2n+1 : n \ge 0\}\),
\(\sigma(B) = \{2m+1 : m \ge 0\}\).
In the tensor product \(H = H_1 \otimes H_2\), the operators
\(A \otimes I\) and \(I \otimes B\) are commuting, self-adjoint, and bounded below.

\medskip
\noindent\textbf{Partial contraction.}
Define a bilinear pairing that contracts the second tensor factor:
for simple tensors \(u \otimes \phi, v \otimes \psi \in H_1 \otimes H_2\),
\[
(u \otimes \phi) \bullet_{H_2} (v \otimes \psi)
:= \langle \phi, \psi \rangle_{H_2}\, u \in H_1,
\]
and extend by linearity and continuity to a bounded bilinear map
\[
\bullet_{H_2} : (H_1 \otimes H_2) \times (H_1 \otimes H_2) \to H_1.
\]
This operation “contracts over” \(H_2\) while preserving the \(H_1\)-component.

\medskip
\noindent\textbf{Definition of the family.}
For \(n \ge 1\), define a family of multilinear maps
\[
T_n : H^n \longrightarrow H_1
\]
recursively using \(\bullet_{H_2}\).  
For \(n = 2\),
\[
T_2(f,g)
:= (A \otimes I)f \bullet_{H_2} (I \otimes B)g.
\]
For general \(n\), let \(P(X_1,Y_1,\dots,X_n,Y_n)
= \sum_{\mathbf{k},\mathbf{l}} c_{\mathbf{k},\mathbf{l}}
  X_1^{k_1}Y_1^{l_1}\cdots X_n^{k_n}Y_n^{l_n}\)
be a polynomial in commuting variables.  
Then on a dense invariant domain \(\mathcal{D} \subset H\) (for instance, the finite span of Hermite tensors), define
\[
T_n(f_1,\dots,f_n)
:= \sum_{\mathbf{k},\mathbf{l}} c_{\mathbf{k},\mathbf{l}}
  \Bigl[
    (A^{k_1}\otimes B^{l_1})f_1
    \bullet_{H_2}
    \cdots
    \bullet_{H_2}
    (A^{k_n}\otimes B^{l_n})f_n
  \Bigr],
\]
where the iterated contraction is applied left-to-right.  
This defines a densely defined multilinear operator with values in \(H_1\).

\medskip
\noindent\textbf{Domain and boundedness.}
Since \(A,B\) are unbounded, powers \(A^k, B^l\) act on domains \(D(A^k), D(B^l)\).  
To ensure rigor, one may adopt one of the following standard interpretations:
\begin{enumerate}
    \item \emph{Domain approach:}  
    Work on a dense invariant domain \(\mathcal{D} = \mathrm{span}\{h_m \otimes h_n\}\) of Hermite tensor eigenfunctions.  
    Each \(T_n\) is well-defined and continuous when each input is equipped with the graph norm of the corresponding operator powers.

    \item \emph{Bounded transform approach:}  
    Replace \(A, B\) by bounded functions such as \((I + A)^{-1}\), \((I + B)^{-1}\) or by their Cayley transforms.  
    Then each \(T_n\) extends to a globally bounded multilinear operator on \(H\), compatible with the Banach-enriched multicategory \(\mathbf{HilbMult}\).
\end{enumerate}

\medskip
\noindent\textbf{Spectral and polynomial properties.}
On either version (domain or bounded transform), the family
\(\mathcal{T} = \{T_n\}\) satisfies:
\begin{itemize}
    \item \emph{Spectral locality:}
    The joint spectrum of \(A \otimes I\) and \(I \otimes B\) is
    \(\sigma(A) \times \sigma(B)\); thus any polynomial action
    has spectrum contained in this product set.

    \item \emph{Polynomial compatibility:}
    For polynomials \(p_j\),
    \[
    T_n\bigl(p_1(A \otimes I, I \otimes B)f_1, \dots, p_n(A \otimes I, I \otimes B)f_n\bigr)
    = (p_1 \cdots p_n)(A \otimes I, I \otimes B)\,
      T_n(f_1, \dots, f_n).
    \]

    \item \emph{Uniform boundedness:}
    In the bounded-transform setting, all \(T_n\) are bounded on \(H\);
    in the domain setting, they are continuous in the induced graph-topology.
\end{itemize}

\medskip
\noindent\textbf{Concrete instance.}
For the polynomial \(P(X_1,Y_1,X_2,Y_2) = X_1 Y_2\),
\[
\mathcal{F}_{(A,B),\mathcal{T}}(P)(f,g)
= T_2(f,g)
= (A \otimes I)f \bullet_{H_2} (I \otimes B)g \in H_1,
\]
which corresponds to applying \(A\) to the first component of \(f\),
\(B\) to the second component of \(g\), and contracting over \(H_2\).  
For the constant polynomial \(Q \equiv 1\),
\[
\mathcal{F}_{(A,B),\mathcal{T}}(Q)(f,g)
= f \bullet_{H_2} g,
\]
the contraction reduces to the pure tensor pairing.

\medskip
This example illustrates that the family \(\mathcal{T} = \{T_n\}\)
realizes a coherent system of multilinear operator functionals over
commuting self-adjoint operators in a tensor product Hilbert space,
and can be interpreted as a Banach-enriched multicategorical model
of tensorized operator calculus.
\end{example}

\section{Programmatic Vision: The Categorical Spectral Architecture}

The $\mathbf{HilbMult}$ multiclass construction presented here and the associated spectral-functional theory represent the basic step of a long-term and integrated research program. This work introduces \textbf{Categorical Spectral Geometry (CSA)} - A comprehensive framework designed to synthesize operator theory higher-order categorical structures and non-commutative geometry into a unified formalism for complex systems.

CSA is not seen as a rigid plan but rather as a dynamic and evolving research agenda.  Its development continues through the development of interrelated conceptual layers each of that is designed to address a fundamental aspect of the universal theory of operators. The ultimate goal is to create a common definitive language for the syntax semantics and geometry of operators thereby enabling a deep synthesis of logic physics and computation.

\subsection{The Conceptual Axes of Development}

The long-term trajectory of the CSA is organized along five interdependent conceptual axes. These axes represent the core pillars upon which the architecture will be erected.

\begin{enumerate}
    \item \textbf{The Syntactic Axis: Operadic Coherence as a Universal Language.}
    The semantic universe of $\mathbf{HilbMult}$ demands a rigorous syntactic counterpart. This axis focuses on the construction of a \emph{Synergy Operad}, which will serve as a universal grammatical framework for composing operator networks. A central milestone will be the establishment of a \emph{Coherence Theorem}, guaranteeing that diagrammatically equivalent compositions yield identical morphisms, thereby translating intuitive graphical reasoning into formally verifiable categorical statements~\cite{Chang2025CompositionCoherence}.

    \item \textbf{The Algebraic Axis: The Categorical Algebra of Operators.}
    This axis seeks to ground classical operator algebra within the multicategorical paradigm. The goal is to develop a theory of \emph{C\(^*$-algebraic operability} within structures like $\mathbf{HilbMult}$, formalizing how fundamental operations—such as tensor products, commutators, and adjoints—interact functorially across categorical layers. The outcome will be a deep \emph{Categorical Algebra of Operators}, revealing coherence laws as the foundation for algebraic identities.

    \item \textbf{The Structural Axis: Higher Dualities and Physical Constraints.}
    Here, we ascend to the study of intrinsic symmetries and limitations. By developing a comprehensive theory of \emph{$n$-adjoints} in multicategories, this axis aims to recast fundamental physical principles—such as unitarity, causality, and the no-cloning theorem—not as external postulates, but as emergent properties expressed via \emph{operadic ideals} or structural identities. This promises a paradigm where the architecture of physical laws is derived from categorical geometry.

    \item \textbf{The Analytical Axis: A Functorial Calculus for Spectral Dynamics.}
    This axis bridges our categorical framework back to the core concerns of functional analysis. The mission is to construct a \emph{Functorial Calculus of Operators} that seamlessly integrates analytic notions of spectrum, norm, and approximation. This involves developing tools like a categorified \emph{Goodwillie calculus} to obtain quantitative spectral bounds and stability criteria for complex networks, thereby creating a powerful analytical toolkit native to the categorical setting.

    \item \textbf{The Geometric Axis: Conquering Noncommutativity via Spectral Stacks.}
    The ultimate objective of the CSA is to geometrize noncommutativity. This axis posits that the spectrum of a noncommuting family of operators is inherently a \emph{Spectral Stack}—a rich geometric object encoding its algebraic relations. The grand challenge is to develop the \emph{Analytic–Geometric Toolkit} required to construct and analyze these stacks, culminating in a \emph{Categorified Spectral Duality} principle. This principle will reveal the classical spectral theorem as a degenerate case of a far more general, geometrically profound reality.
\end{enumerate}

\subsection{A Living Research Ecosystem}

The \textbf{Categorical Spectral Architecture} is an open and adaptive research ecosystem. The five axes outlined above provide the central compass for its evolution, but their interactions are expected to generate unforeseen directions and hybrid subfields. Each subsequent contribution within this program will illuminate a distinct facet of the architecture, while simultaneously reinforcing the coherence and power of the whole.

This paper lays the cornerstone. The edifice to be built upon it promises nothing less than a unified categorical foundation for modern operator theory, with transformative potential for mathematics, quantum physics, and information science. We envision the CSA not merely as a technical framework, but as a new language for understanding the architecture of complexity itself.

\bibliographystyle{IEEETran}
\bibliography{HilbMult_Banach_Enriched_Multicategory_for_Operator_Algebras_Bib}

\end{document}